\documentclass[11pt,reqno]{amsart}

\usepackage{hyperref}

\allowdisplaybreaks[4]

\usepackage{amsmath}
\usepackage{amssymb}
\usepackage{mathrsfs}
\usepackage{amsthm}
\usepackage{bm}
\usepackage{graphicx}
\usepackage{booktabs}
\usepackage{float}
\usepackage{geometry}
\usepackage{color}

\theoremstyle{plain}
\newtheorem{lemma}{Lemma}[section]
\newtheorem{theorem}[lemma]{Theorem}

\newtheorem{proposition}[lemma]{Proposition}
\newtheorem{corollary}[lemma]{Corollary}
\newtheorem{assumption}{Assumption}

\newcommand{\bbE}{\mathbb E}

\newcommand{\bbN}{\mathbb N}
\newcommand{\bbR}{\mathbb R}

\begin{document}
\title[Fully discrete schemes   for SHE with L\'evy noise]{On structure preservation for fully discrete finite difference schemes of stochastic heat equation with  L\'evy space-time white noise}
\author{Chuchu Chen, Tonghe Dang, Jialin Hong}
\address{LSEC, ICMSEC,  Academy of Mathematics and Systems Science, Chinese Academy of Sciences, Beijing 100190, China,
\and 
School of Mathematical Sciences, University of Chinese Academy of Sciences, Beijing 100049, China}
\email{chenchuchu@lsec.cc.ac.cn; dangth@lsec.cc.ac.cn; hjl@lsec.cc.ac.cn}
\thanks{This work is funded by the National key R\&D Program of China under Grant (No.
2020YFA0713701), National Natural Science Foundation of China (No. 12031020), and by Youth Innovation Promotion Association CAS, China.}
\begin{abstract}  
 This paper investigates the structure preservation and convergence analysis of 
a class of fully discrete finite difference schemes for the stochastic heat equation driven by  L\'evy space-time white noise. The novelty 
lies in the simultaneous preservation of 
 intrinsic structures for  the exact solution, in particular  the weak intermittency  of moments and the regularity of c\`adl\`ag path in negative fractional Sobolev spaces. The key in the proof is the detailed  analysis of  technical estimates  for  discrete Green functions of the numerical solution. This analysis is also crucial   in  establishing  the mean-square convergence of the schemes with  orders of  almost $\frac12$ in space and almost $\frac14$ in time.  \end{abstract}

\maketitle

\section{Introduction}
In this paper, we study the structure preservation and  convergence analysis for a class of  fully discrete schemes  of   the stochastic heat equation (SHE) with periodic boundary condition  
\begin{align}
\begin{cases}
\partial_t u(t,x)=\partial^2_x u(t,x)+\sigma(u(t,x))\dot{\Lambda}(t,x),&\vspace{1 ex}\\
u(t,0)=u(t,1),\quad t\in [0,T], & \vspace{1 ex}\\
u(0,x)=u_0(x),\quad 0\leq x\leq 1, &
\end{cases}\label{she}
 \end{align}
 where $\sigma: \mathbb{R} \rightarrow \mathbb{R}$ is a globally Lipschitz function, and $u_0$ is a bounded, non-negative, non-random, and measurable function.
Here, $\dot{\Lambda}(t,x), t\ge0,x\in[0,1]$ denotes the  L\'evy space-time white   
 noise, which is the distributional derivative of a L\'evy sheet in $(1+1)$ parameters, 
on some complete filtered probability space $\big(\Omega,\mathcal{F},  \{\mathcal{F}_t \}_{t\ge 0},\mathbb{P}\big)$. Precisely, we assume that $\Lambda$ takes the form
\begin{align}\label{Lambda}
\Lambda(d t,d x)=bdtdx+\int_{\{|z|\leq 1\}}z\tilde {\mu}(dt,dx,dz)+\int_{\{|z|>1\}}z\mu(dt,dx,dz),
\end{align}
where $b\in\bbR,$ 
$\mu$ is a Poisson measure on $(0,\infty)\times [0,1]\times \mathbb R$ with intensity measure $\nu( dt,dx,dz)=dtdx\lambda(dz),$ and $\tilde \mu$ is the compensated version of $\mu.$ Here, 
$\lambda$ is a L\'evy measure satisfying 
\begin{align*}
\lambda(\{0\})=0,\quad \int_{\bbR}(1\wedge  |z|^2)\lambda(dz)<\infty.
\end{align*}
We assume that $\lambda$ is not identically zero. The L\'evy noise, as a typical example of the non-Gaussian noise, has important  applications in modeling stochastic processes characterized by jumps or sudden events. 
The SHE  with L\'evy  noise is widely applied to model more complex  phenomena such as financial market crashes, abrupt phase transitions, and  neural spiking patterns. 
We refer to  e.g. \cite{Applebook,Hau08,evolution} and references therein for details.

It is known that there are distinct differences for the solution of the SHE with the Gaussian noise and that with the L\'evy noise. Below we focus on the intermittency-type property of moments and the path property to illustrate the difference. 
 
  \textit{Gaussian space-time  white noise case.} The solution of the SHE with  Gaussian space-time  white noise  exhibits certain regular property    with finite moments of all orders,  namely,   \begin{align*}
ce^{ct}\leq \inf_{x\in[0,1]}\mathbb E[|u(t,x)|^p]\leq \sup_{x\in[0,1]}\mathbb E[|u(t,x)|^p]\leq Ce^{Ct},\quad p> 1  
\end{align*}
with some constants  $c,C>0$.  As a result,  the solution  is weakly intermittent of  all orders $p> 1;$ see e.g. \cite{dang2022} and references therein.   
Moreover, the solution has  continuous sample paths 
  in $H^r$ with $r<\frac12$,  namely, for any small constant $\epsilon>0,$ 
\begin{align}\label{path_ex}
\mathbb E[\|u(t)-u(s)\|_{H^r}^{2p}]\leq C|t-s|^{p(\frac12-r-\epsilon)},\quad p\ge 1,
\end{align}
see e.g. \cite[Theorem 2.31]{Kruse}.    Here, $H^r,r\in\mathbb R$ is the usual Sobolev space (see Section \ref{prelimi}).
Based on the Kolmogorov  continuity theorem,   \eqref{path_ex}  implies that the solution process  $u$ admits a continuous modification of sample paths in $H^r$  with $r<\frac12$. 

\textit{L\'evy space-time white noise case.} In contrast to  Gaussian noise, L\'evy noise has the unique   ability to capture discontinuities   and irregular changes in complex phenomena, often leading to    the non-existence of higher moments for the associated  stochastic processes. For the SHE \eqref{she} driven by   L\'evy space-time white noise, 
the solution has finite moments only up to the order $3-\epsilon$ with any small constant $\epsilon>0$. 
To be specific, under the moment boundedness  condition on the noise, it is shown that the unique  solution  satisfies
\begin{align}\label{mo_ex1}
ce^{ct}\leq \inf_{x\in[0,1]}\mathbb E[|u(t,x)|^p]\leq \sup_{x\in[0,1]}\mathbb E[|u(t,x)|^p]\leq Ce^{Ct},\quad p\in(1,3) 
\end{align}
with some constants $c,C>0.$ This  indicates  that 
\begin{align}\label{weak_ex}
0<\underline{\gamma}(p)\leq \bar{\gamma}(p)<\infty,\quad p\in(1,3),
\end{align}
which means that the solution is weakly intermittent of order $p\in(1,3);$ see Proposition \ref{ex_intermittent} for details. Here, $\bar{\gamma}(p)$ and $\underline{\gamma}(p)$ are $p$th upper and lower moment Lyapunov exponents, respectively, defined as  
\begin{align*}&\bar{\gamma}(p):=\limsup\limits_{t\to\infty}\frac{\sup_{x\in[0,1]}\log \mathbb E[|u(t,x)|^p]}{t},
\quad \underline{\gamma}(p):=\liminf\limits_{t\to\infty}\frac{\inf_{x\in[0,1]}\log \mathbb E[|u(t,x)|^p]}{t}.
\end{align*} Moreover, the discontinuity in particular the jump of the L\'evy noise will create a Dirac mass for the solution,  which results that the  solution cannot be expected to have a c\`adl\`ag  modification (right-continuous with finite left-limit) in any positive Sobolev spaces. 
A key feature to characterize the c\`adl\`ag path is as follows: for any $t\in(0,T]$ and $h\in(0,t\wedge 1),$ 
\begin{align}\label{path_ex1}
\mathbb E[|\textrm{osc}_r(u(t+h),u(t))\textrm{osc}_r(u(t),u(t-h))|^2]\leq Ch^{1+\delta}
\end{align} 
holds for some    small constant $\delta>0,$ where we define    the oscillation $\textrm{osc}_r(u(t+h),u(t)):=\|u(t+h)-u(t)\|_{H^r}.$ 
It is proved that \eqref{path_ex1} holds when  $r<-\frac12$ for the L\'evy space-time white noise case under the boundedness condition on the diffusion coefficient. 
We refer to  e.g. \cite{Applebaum,evolution} for the study of the well-posedness of the exact solution, and to e.g. \cite{Berger,SPDEAC19,Chong} for  the investigation of the intermittency-type property of moments and the path property  for  the exact solution of SHE with  L\'evy noise.

There have been fruitful works on the study of numerical methods 
for the SHE with Gaussian space-time white noise. By contrast, the  L\'evy noise  case remains  relatively underdeveloped and there are only few works on this aspect. 
For example, \cite{Hau08} investigates  the accuracy and approximation of stochastic partial differential equations with space-time Lévy noise using a finite element method in space and an implicit Euler scheme in time;  \cite{Hau12} presents various discretization methods to  accurately simulate  jumps induced by Lévy noise, alongside an analysis of an implicit time-discretization method. 
The aim of this paper is to study the preservation of both  the weak intermittency of moments \eqref{weak_ex} and the path property \eqref{path_ex1} of   
a class of fully discrete schemes for SHE \eqref{she} with  L\'evy space-time  white noise.  

To this end, we apply the finite difference method  in space and further the $\theta$-scheme with $\theta\in[0,1]$ in time to obtain a class of  fully discrete schemes. The numerical solution  is  c\`adl\`ag in space or in time when one of the variables is  fixed. These  schemes have  the  mild formulation with explicit  expressions of    discrete Green functions, 
which is fundamental to   analyzing  the preservation of  intrinsic structures of the exact solution.  
With  technical   estimates of  discrete Green functions, we show that 
the fully discrete schemes  inherit the weak intermittency of the exact solution for moments of  order $p\in(1,3)$, namely, 
\begin{align*}
0<\underline{\gamma}^{n,\tau}(p)\leq \bar{\gamma}^{n,\tau}(p)<\infty,\quad p\in(1,3),
\end{align*}
where $\bar{\gamma}^{n,\tau}(p),\underline{\gamma}^{n,\tau}(p)$ are discrete versions of $p$th upper and lower moment Lyapunov exponents, respectively. 
  In addition, by presenting  the relation between the norm in the negative Sobolev space and its discrete counterpart, we prove that the mild solution of the  fully discrete scheme   preserves the relation \eqref{path_ex1} uniformly with respect to the discretization parameters, namely, for any $t\in(0,T]$ and  $h\in(0,t\wedge 1),$  
  \begin{align*}
\sup_{n,\tau}\mathbb E[|\textrm{osc}_r(u^{n,\tau}(t+h),u^{n,\tau}(t))\textrm{osc}_r(u^{n,\tau}(t),u^{n,\tau}(t-h))|^2]\leq Ch^{1+\delta}
\end{align*} 
holds for some $\delta>0$ and  for all  $r<-\frac12$.
This result yields  that the numerical solution $\{u^{n,\tau}\}_{n,\tau}$ is weakly relatively compact in the Skorohod space $D([0,T];H^r)$ with $r<-\frac12.$

Furthermore, the convergence of the fully discrete schemes is also carefully analyzed. The prerequisite  is the  error estimates between  the discrete Green functions and the Green function of the exact solution. We show that the discrete Green functions converge to the exact one in the integral sense with certain orders.  Then we prove that the fully discrete scheme achieves  the  mean-square convergence orders of  almost $\frac12$ in space and almost $\frac14$ in time. 
In addition, we also present some discussions on the  the more general case that the noise is of infinite variance. In this case,   we  introduce a noise truncation skill and obtain a truncated numerical solution for the fully discrete scheme,  
which still possesses the weak intermittency and path property.    
We prove that the truncated numerical solution converges almost surely to the exact solution.

This paper is organized as follows. In Section \ref{prelimi}, we give some preliminaries for  the exact solution of \eqref{she}, including the well-posedness,  the weak intermittency, and the path property. 
In Section \ref{sec2}, we introduce the  fully discrete schemes of  \eqref{she}, and then prove the preservation of both the  weak intermittency and the path property of the exact solution. 
In Section \ref{sec_con},   we  show the   convergence order  of the fully  discrete scheme, and also give some discussions for  the infinite variance noise case. 
Section \ref{sec5} is devoted to the proof of  error estimates between  discrete Green functions and the Green function of the exact solution. 

Throughout this paper, we use $C$ to denote a positive constant which may not be the same in each occurrence. More specific constants which depend on certain parameters $a,b$ are numbered as $C(a,b)$.

\section{Preliminaries}\label{prelimi}
This section is devoted to presenting some preliminaries for the exact solution of \eqref{she}, including  the well-posedness,  the weak intermittency, and the path property. 

The mild solution of \eqref{she} has the form of 
\begin{align*}
u(t,x)=\int_0^1G(t,x,y)u_0(y)dy+\int_0^t\int_0^1G(t-s,x,y)\sigma(u(s,y))\Lambda(ds,dy). 
\end{align*}
Here, the  function $G$ is known as the Green function defined by \begin{align*}
G(t,x,y)=\frac{1}{\sqrt{4\pi t}}\sum_{m=-\infty}^{+\infty}e^{-\frac{(x-y-m)^2}{4t}}, 
\end{align*}
which  also has  the  spectral decomposition $G(t,x,y)=\sum_{m=-\infty}^{+\infty}e^{-4\pi^2m^2t}e^{2\pi\mathbf im(x-y)}$ for $ t>0,x,y\in[0,1];$ see e.g. \cite{dang2022}.
The  mild solution of \eqref{she} with the L\'evy space-time noise \eqref{Lambda} is well-posed, which is stated in the following proposition.  The  proof is   similar to \cite[Proposition 2.1]{SPDEAC19} and thus is omitted.  
\begin{assumption}\label{assum_noise} Assume that $m_{\lambda}(p):=\int_{\mathbb R}|z|^p\lambda(dz)<\infty$ for $p\in[1,3)$.
\end{assumption}
\begin{proposition}  \label{stop_solution} Under Assumption \ref{assum_noise}, 
there exists a unique mild solution of \eqref{she} satisfying   
$
\sup_{x\in[0,1]}\mathbb E[|u(t,x)|^p]\leq Ce^{Ct}, p\in[1,3)
$ with some constant $C>0.$
\end{proposition}

 To present the weak intermittency of the exact solution, we introduce the following assumption on the    coefficient and the noise. Denote $L_{\sigma}:=\sup_{x\neq y,\,x,y\in\bbR}\big|\frac{\sigma(x)-\sigma(y)}{x-y}\big|$, $J_0:=\inf_{x\in\bbR\setminus \{0\}}\big|\frac{\sigma(x)}{x}\big|$.  

\begin{assumption}\label{assump1} 
 Assume  $L_{\sigma}>0,$ $J_0>0,$ $u_0\equiv c>0,$ and  $b=-\int_{\{|z|>1\}}z\nu(dz)$.   
\end{assumption}

\begin{proposition}\label{ex_intermittent}
Let Assumptions \ref{assum_noise} and  \ref{assump1} hold. Then the mild solution of \eqref{she} is weakly intermittent of order $p\in(1,3)$, i.e.,  $0<\underline{\gamma}(p)\leq \bar{\gamma}(p)<\infty,\; p\in(1,3).$ \end{proposition}
\begin{proof}
From Proposition \ref{stop_solution}, we can obtain  the intermittent upper bound: $\bar{\gamma}(p)<\infty$ for $p\in[1,3).$ 
To prove the intermittent lower bound, we apply \cite[Lemmas 5.4 and 3.4]{Chong} to derive that for $p\in(1,2), 
 $\begin{align*}
&\mathbb E[|u(t,x)|^p]\ge C|u_0|^p+C\mathbb E\Big[\Big|\int_0^{t}\int_0^1G(t-s,x,y)\sigma(u(s,y))\Lambda (ds,dy)\Big|^p\Big]\\
&\ge C+C\int_0^{t}\int_0^1|G(t-s,x,y)|^pdy\inf_{y\in[0,1]}\mathbb E[|u(s,y)|^p]ds.
\end{align*}
Noting  that $(\sum_{n=1}^{\infty}a_n)^p\ge \sum_{n=1}^{\infty}a_n^p$ for $a_n\ge 0,p\ge 1,$ 
 we have  
 \begin{align*}
\int_0^1|G(t,x,y)|^pdy&\ge (\sqrt{4\pi t})^{-p} \sum_{m=-\infty}^{\infty}\int_{x-m-1}^{x-m}e^{-\frac{pz^2}{4t}}dz 
=p^{-\frac12}(4\pi)^{\frac{1-p}{2}}t^{\frac{1-p}{2}}.
\end{align*}
Hence, we arrive at
$\inf_{x\in[0,1]}\mathbb E[|u(t,x)|^p]
\ge C+C\int_0^{t}(t-s)^{\frac{1-p}{2}}\inf\limits_{y\in[0,1]}\mathbb E[|u(s,y)|^p]ds.$ 
Multiplying $e^{-\beta t}$ on both sides gives
\begin{align*}
e^{-\beta t}\inf_{x\in[0,1]}\mathbb E[|u(t,x)|^p]\ge Ce^{-\beta t}+C\int_0^t(t-s)^{\frac {1-p}{2}}e^{-\beta(t-s)}e^{-\beta s}\inf_{y\in[0,1]}\mathbb E[|u(s,y)|^p]ds.
\end{align*}
Noticing $C\int_0^{\infty}s^{\frac{1-p}{2}}e^{-\beta s}ds=C\Gamma(\frac{3-p}{2})\beta^{-\frac{3-p}{2}},$ we can take $\beta=(C\Gamma(\frac{3-p}{2}))^{\frac{2}{3-p}}$ so that $Cs^{\frac {1-p}{2}}e^{-\beta s}$
is a probability function on $s>0.$
Applying the renewal theorem (see e.g. \cite[Theorem V.7.1]{renewal}) and the property of the super-solution (see e.g. \cite[Theorem 7.11]{uppersolution}) gives that for  sufficiently large $t$,  
\begin{align*}
e^{-\beta t}\inf_{x\in[0,1]}\mathbb E[|u(t,x)|^p]\ge \frac{\int_0^{\infty}Ce^{-\beta s}|u_0|^p ds}{\int_0^{\infty}s e^{-\beta s}s^{\frac{1-p}{2}}ds}= C|u_0|^p\beta^{\frac{3-p}{2}}(\Gamma(\frac{5-p}{2}))^{-1}.
\end{align*}
This leads to 
$
\underline{\gamma}(p)\ge \beta>0
$ for $p\in(1,2). $ 
Combining the convexity of the map $p\mapsto \underline{\gamma}(p)$ 
yields that $\underline{\gamma}(p)>0$ for all $p\in(1,3),$ which 
finishes the proof. 
\end{proof}

To show the path property of the exact solution,  we first give a brief introduction to the usual Sobolev space $H^r,r\in\mathbb R.$ It is known that  $\{e_k(x):=e^{2\pi \mathbf{i}kx},x\in[0,1]\}_{k\in\mathbb N}$ forms an orthonormal basis of $H:=L^2(0,1)$ with the periodic boundary condition. Each  function $v\in H$ can be expanded in an exponential Fourier series: $
v(x)=\sum_{k=0}^{\infty}c_k(v)e^{2\pi \mathbf ikx}$ with $ c_k(v)=\int_0^1v(y)e^{-2\pi \mathbf iky}dy.
$ 
 The norm on the Sobolev space $H^r,r\in\bbR$ is  defined as 
$\|v\|_{H^r}:=\big(\sum_{k=0}^{\infty}(1+4\pi^2 k^2)^r|c_k(v)|^2\big)^{\frac12},$ see e.g. \cite[Section 2.2]{Sobolov2015} for more details. 

\begin{proposition}\label{path_pro}
 Let Assumption \ref{assum_noise} hold and $\sigma$ be bounded. Then for any $t\in(0,T] $ and $h\in(0,1\wedge t),$
\begin{align}\label{path_ex2}\mathbb E\Big[|\textrm{osc}_r(u(t+h),u(t))\textrm{osc}_r(u(t-h),u(t))|^2\Big]\leq Ch^{1+\delta}  
\end{align} 
holds for some $\delta>0$ and any   $r<-\frac12.$ 
\end{proposition}
The proof can be found in \cite{SPDEAC19}, where  the mild solution of \eqref{she} is also proved to admit a  c\`adl\`ag version in $H^{r}$ with $r<-\frac12.$  

\section{Fully  discrete scheme and structure preservation}\label{sec2}
 In this section, we introduce a class of fully  discrete schemes  of  \eqref{she}, whose spatial direction is   based on the finite difference method and temporal direction  is the $\theta$-scheme $(\theta\in[0,1])$.  We prove that the numerical solution can simultaneously  preserve  the weak intermittency and the path property of the exact solution.   

\subsection{Fully discrete scheme}
Introduce the uniform partition on the spatial domain $[0,1]$ with step size $\frac{1}{n}$ for a fixed integer $n\ge 3$.
Let $u^n(t,\frac{k}{n})$ be the approximation of $u(t,\frac{k}{n})$, $k=0,1,\ldots,n-1$.
The spatial semi-discretization  based on the finite difference method is given by:
\begin{equation}\label{fdm}
\begin{aligned}
\begin{cases}
du^n(t,\frac{k}{n})=n^2(u^n(t,\frac{k+1}{n})-2u^n(t,\frac{k}{n})+u^n(t,\frac{k-1}{n}))dt+n\sigma(u^n(t,\frac{k}{n}))\Lambda^{n,k}(dt),\\
u^n(t,0)=u^n(t,1),\quad u^n(t,-\frac{1}{n})=u^n(t,\frac{n-1}{n}),\quad t\ge 0,\\
u^n(0,\frac{k}{n})=u_0(\frac{k}{n}),\quad k=0,1,\ldots, n-1,
\end{cases}
\end{aligned}
\end{equation}
where $\Lambda^{n,k}(dt):=\Lambda(dt,[\frac kn,\frac{k+1}{n}))=\int_{\frac{k}{n}}^{\frac{k+1}{n}}\Lambda(dt,dx).$
Fix the uniform time step size $\tau\in(0,\frac12)$. By using the $\theta$-scheme to discretize \eqref{fdm}, we obtain the following fully   discretize scheme:
\begin{align}\label{fullscheme}
\begin{cases}
u^{n,\tau}(t_{i+1},x_j)=u^{n,\tau}(t_i,x_j)+(1-\theta)\tau \Delta_nu^{n,\tau}(t_i,\cdot)(x_j)+\theta \tau \Delta_nu^{n,\tau}(t_{i+1},\cdot)(x_j) \\
\qquad \qquad \qquad \quad+ n \sigma(u^{n,\tau}(t_i,x_j))\square_{n,\tau}\Lambda(t_i,x_j), \\
u^{n,\tau}(t_i,0)=u^{n,\tau}(t_i,1),\quad u^{n,\tau}(t_i,-\frac{1}{n})=u^{n,\tau}(t_i,\frac{n-1}{n}),\quad i=0,1,\ldots, \\
u^{n,\tau}(0,x_j)=u_0(x_j),\quad j=0,1,\ldots, n-1,
\end{cases}
\end{align} 
where $u^{n,\tau}$ is an approximation of $u^n$, $ t_i:=i\tau,x_j:=\frac{j}{n}$, and
\begin{equation*}
\begin{aligned}
&\Delta_nu^{n,\tau}(t_i,\cdot)(x_j):=n^2(u^{n,\tau}(t_i,x_{j+1})-2u^{n,\tau}(t_i,x_j)+u^{n,\tau}(t_i,x_{j-1})),\vspace{1 ex}\\
&\square_{n,\tau}\Lambda (t_i,x_j):=\Lambda([t_i,t_{i+1}),[x_j,x_{j+1})).
\end{aligned}
\end{equation*}
Similar to \cite[Eq. (17)]{dang2022},  the mild form of the fully discrete scheme is given by: 
\begin{align}\label{mild full}
u^{n,\tau}(t,x)=&\;\int_{0}^{1}G^{n,\tau}_1(t,x,y)u_0(\kappa_n(y))\,dy\nonumber\\
&+\int_{0}^{t}\int_{0}^{1}G^{n,\tau}_2(t-\kappa_{\tau}(s)-\tau,x,y)\sigma (u^{n,\tau}(\kappa_{\tau}(s),\kappa_n(y)))\Lambda(d s,d y),
\end{align}
almost surely for every $t=i\tau, i\ge 1,\,x\in [0,1]$, where $\kappa_{\tau}(s):=[\frac{s}{\tau}]\tau$ and   
$\kappa_n(y):=\frac{[ny]}{n}$ with  $[\cdot]$ being the greatest integer function. Here, 
the fully discrete Green functions are defined as 
\begin{align*}
&G^{n,\tau}_1(t,x,y):=\sum_{l=0}^{n-1}(R_{1,l}R_{2,l})^{[\frac{t}{\tau}]}e_l(\kappa_n(x))\bar{e}_l(\kappa_n(y))\mathbf 1_{\{t\ge 0\}},\\
&G^{n,\tau}_2(t,x,y):=\sum_{l=0}^{n-1}(R_{1,l}R_{2,l})^{[\frac{t}{\tau}]}R_{1,l}e_l(\kappa_n(x))\bar{e}_l(\kappa_n(y))\mathbf 1_{\{t\ge 0\}},
\end{align*}
where $R_{1,l}:=(1-\theta \tau\lambda^n_l)^{-1}, R_{2,l}:=1+(1-\theta)\tau \lambda^n_l$  with $\lambda^n_l:=-4n^2\sin^2(\frac{l\pi}{n})$,  
$e_l(x)=e^{2\pi \mathbf{i}lx}$, and $\bar{e}_l(\cdot)$ represents the complex conjugate of $e_l(\cdot).$ Let $f_l\in\mathbb C^n,l=0,\ldots,n-1$,   whose $k$th component is $[f_l]_k:=\frac{1}{\sqrt{n}}e^{2\pi \mathbf{i}l\frac{k}{n}},k=0,1,\ldots, n-1$. Then $\{f_l,l=0,1,\ldots, n-1\}$ forms an orthonormal basis in $\mathbb{C}^n$. 
When $t\in[t_i,t_{i+1}),$ we define $u^{n,\tau}(t,x)=u^{n,\tau}(t_i,x)$. Then $u^{n,\tau}(\cdot,x)$ is right-continuous and has left-limit for $x\in[0,1]$.

We make the following assumption on the spatial step size $\frac{1}{n}$ and the temporal step size $\tau$ when $\theta$ takes different values, to ensure the well-posedness of the fully discrete Green functions and the numerical solutions. We also refer to \cite{dang2022} for more details about this assumption. 
\begin{assumption}\label{Assumption3}
(\romannumeral1) For $0\leq \theta<\frac{1}{2},$ suppose 
$n^2\tau\leq r<\frac{1}{2-4\theta}$ with some constant $r>0$.

(\romannumeral2) For $\theta=\frac{1}{2},$ suppose 
$n^2\tau\leq \frac{1}{\epsilon}-\frac{1}{2}$ with any fixed  $\epsilon\in(0,\frac{1}{2}).$

(\romannumeral3) For $\frac{1}{2}<\theta\leq 1$, there is no coupled requirement for $n,\tau.$
\end{assumption}

\subsection{Preservation of the weak intermittency}
In this subsection, we show that the numerical solution of the fully discrete scheme inherits  the weak intermittency of the exact solution.  
Define the discrete versions of the $p$th upper and lower moment Lyapunov exponents as $\bar{\gamma}^{n,\tau}(p):=\limsup\limits_{i\to\infty}\frac{\sup_{x\in[0,1]}\log \mathbb E[|u^{n,\tau}(t_i,\kappa_n(x))|^p]}{t_i}$ and $\underline{\gamma}^{n,\tau}(p):=\liminf\limits_{i\to\infty}\frac{\inf_{x\in[0,1]}\log\mathbb E[|u^{n,\tau}(t_i,\kappa_n(x))|^p]}{t_i},$ respectively.    
 To prove the weak intermittent of moments for numerical solution, we need  the integrability   property of the discrete Green function and an  inverse Gr\"onwall inequality, which are stated in the following two lemmas, respectively. 
\begin{lemma}\label{propG}
 Under Assumption \ref{Assumption3}, we have that for $x\in[0,1]$,   $p\in[1,3),$ and $\beta>0,$
$$
\int_0^{\infty}\int_0^1|G^{n,\tau}_2(s,x,y)|^pe^{-\beta ps}dyds\leq \frac{C}{\beta p}+C\Gamma(\tilde{p})(\beta p)^{-\tilde p},
$$
where $\tilde p=\frac{3-p}{2}$ for $p\in[2,3)$ and $\tilde p=\frac{2-p}{2}$ for $p\in[1,2).$
\end{lemma}
\begin{proof}
 When $p\in[2,3)$, applying \cite[Lemma 4.1 (\romannumeral3)]{dang2022} yields
\begin{align*}
&\quad \int_0^{\infty}\int_0^1|G^{n,\tau}_2(s,x,y)|^2dy\sup_{y\in[0,1]}|G^{n,\tau}_2(s,x,y)|^{p-2}e^{-\beta ps}ds\\
&\leq C\int_0^{\infty}(1+\frac{1}{\sqrt{s}})(1+\sum_{j=1}^{n-1}|R_{1,j}R_{2,j}|^{[\frac {s}{\tau}]}|R_{1,j}|)^{p-2}e^{-\beta ps}ds\\
&\leq \leq C\int_0^{\infty}\Big(1+\frac{1}{\sqrt{s}}\Big)^{p-1}e^{-\beta ps}ds\leq \frac{C}{\beta p}+C\Gamma(\frac{3-p}{2})(\beta p)^{\frac{p-3}{2}}, 
\end{align*}
where we use $
1+\sum_{j=1}^{n-1}|R_{1,j}R_{2,j}|^{[\frac{s}{\tau}]}|R_{1,j}|\leq C(1+\frac{1}{\sqrt s}),s>0 
$ whose proof is similar to \cite[Lemma 4.1 (\romannumeral3)]{dang2022}. 
For $p\in[1,2),$ \begin{align*}
&\quad \int_0^{\infty}\int_0^1|G^{n,\tau}_2(s,x,y)|^pe^{-\beta ps}dsdy\leq \int_0^{\infty}(1+C\sum_{j=1}^{n-1}|R_{1,j}R_{2,j}|^{[\frac{s}{\tau}]}|R_{1,j}|)^pe^{-\beta ps}ds\\
&\leq \int_0^{\infty}(1+\frac{1}{\sqrt{s}})^pe^{-\beta ps}ds\leq \frac {C}{\beta p}+C\Gamma(\frac{2-p}{2})(\beta p)^{\frac{p-2}{2}}.
\end{align*}
The proof is completed. 
\end{proof}

For nonnegative numbers $g_k,$ we let $\sum_{0\leq k<0}g_k=0$ and $\prod_{0\leq k<0}(1+g_k)=1.$

\begin{lemma}[Inverse Gr\"onwall inequality]\label{Gronwall}
Let $\{y_k\}_{k\in\mathbb N}$ and $\{g_k\}_{k\in\mathbb N}$ be nonnegative sequences and constant $c_0>0$. If $
y_n\ge c_0+\sum_{0\leq k<n}g_ky_k,\, n\ge 0,
$ then 
$
y_n\ge c_0\prod_{0\leq j<n} (1+g_j),\,n\ge 0. 
$
\end{lemma}
\begin{proof}
We first claim that  
$
y_n\ge c_0+\sum_{0\leq k<n}c_0g_k\prod_{k<j<n}(1+g_j).
$ 
The case of $n=0$ is obvious.  Suppose that $y_n\ge c_0+\sum_{0\leq k<n}g_ky_k$ holds for all $0\leq n<m,$ then we prove the case of $n=m$ by the induction argument. It follows that \begin{align*}
y_m&\ge c_0+c_0\sum_{0\leq k<m}g_k\Big(1+\sum_{0\leq j<k}g_j\prod_{j<i<k}(1+g_i)\Big)\\
&=c_0+\sum_{0\leq j<m}g_jc_0\Big(1+\sum_{j<k<m}g_k\prod_{j<i<k}(1+g_i)\Big).
\end{align*}
Moreover, we have 
\begin{align*}
&\quad \;1+\sum_{j<k<m}g_k\prod_{j<i<k}(1+g_i)\\
&=1+g_{j+1}+g_{j+2}(1+g_{j+1})+\cdots+g_{m-1}(1+g_{j+1})\cdots(1+g_{m-2})
=\prod_{j<i<m}(1+g_i),
\end{align*}
which yields the claim.  Note that
\begin{align*}
c_0\sum_{0\leq k<n}g_k\prod_{k<j<n}(1+g_j)&=c_0\sum_{0\leq k<n}\Big(\prod_{k\leq j<n}(1+g_j)-\prod_{k+1\leq j<n}(1+g_j)\Big)\\
&=c_0\prod_{0\leq j<n}(1+g_j),
\end{align*}
which together with the claim finishes the proof. 
\end{proof}

Introduce the norm on the space of random fields: for $p>0$ and $\beta>0,$
$\mathcal N_{\beta,p}(u):=\sup_{t\ge 0}\sup_{x\in[0,1]}\{e^{-\beta t}\|u(t,x)\|_p\},$ where $\|\cdot\|_p$ denotes the $L^p(\Omega)$-norm.  

\begin{theorem}\label{main1}
Under Assumptions \ref{assum_noise} and \ref{Assumption3}, there exists a unique  mild solution of the fully discrete scheme satisfying that $\sup_{x\in[0,1]}\mathbb E[|u^{n,\tau}(t_i,\kappa_n(x))|^p]\leq Ce^{Ct_i},$ for $i\ge 1,p\in[1,3).$ 
If in addition  Assumption        \ref{assump1} holds,  
then the fully discrete scheme is  weakly  intermittent of order $p\in(1,3)$. 
\end{theorem}
\begin{proof} The proof is split into two steps. 

\textit{Step 1: Intermittent upper bound: $\bar{\gamma}^{n,\tau}(p)<\infty,\,p\in(1,3).$} 
Based on the mild form \eqref{mild full} of the numerical solution, the proof of the existence and uniqueness of the solution is standard by a Picard iteration argument. We refer to  \cite[Proposition 4.1]{dang2022} on  a  similar proof for the Gaussian noise case. When $p\in[2,3),$ applying the  maximal inequality (see e.g. \cite{Levywave, maximal_ineq}), the Minkowski inequality,  and Assumption  \ref{assum_noise}  
gives 
\begin{align*}
&\mathbb E[|u^{n,\tau}(t,x)|^p]
\leq C+C(m_{\lambda}(p))^p\int_0^t\int_0^1|G^{n,\tau}_2(t-\kappa_{\tau}(s)-\tau,x,y)|^p\times\\
&\quad (1+\bbE [\|u^{n,\tau}(\kappa_{\tau}(s),\kappa_n(y))\|_p^p])dsdy+C(m_{\lambda}(2))^{p}\times\\
&\quad \Big(\int_0^t\int_0^1|G^{n,\tau}_2(t-\kappa_{\tau}(s)-\tau,x,y)|^2(1+\|u^{n,\tau}(\kappa_{\tau}(s),\kappa_n(y))\|^2_p)dsdy\Big)^{\frac p2}.
\end{align*}
 When $p\in[1,2),$  using  the maximal inequality again, we obtain    
\begin{align*}
&\quad \bbE[|u^{n,\tau}(t,x)|^p]\\
&\leq C+C\int_0^t\int_0^1|G^{n,\tau}_2(t-\kappa_{\tau}(s)-\tau,x,y)|^p(1+\bbE[|u^{n,\tau}(\kappa_{\tau}(s),\kappa_n(y))|^p])dsdy. 
\end{align*} 
Multiplying $e^{-p\beta t}$ with $\beta>0$ on both sides of the above equation, and combining Lemma \ref{propG} yield  
\begin{align*}
&e^{-\beta pt}\|u^{n,\tau}(t,x)\|^p_p\leq e^{-\beta pt}C+C(1+(\mathcal N_{\beta,p}(u^{n,\tau}))^p)\times\\&\quad \Big[\Big(\int_0^t\int_0^1|G^{n,\tau}_2(t-\kappa_{\tau}(s)-\tau,x,y)|^2e^{-2\beta(t-s)}dyds\Big)^{\frac p2} \\
&\quad +C\int_0^t\int_0^1|G^{n,\tau}_2(t-\kappa_{\tau}(s)-\tau,x,y)|^pe^{-\beta p(t-s)}dyds\Big]\leq C+C(\beta^{-1})(\mathcal N_{\beta,p}(u^{n,\tau}))^p,
\end{align*}
where $C(\beta^{-1})$ is a polynomial of $\beta^{-1}$ satisfying $\lim_{\beta\to\infty}C(\beta^{-1})=\infty.$

The remaining proof is similar to that of \cite[Proposition 4.1]{dang2022} and thus is omitted. As a consequence, we can obtain $\sup_{x\in[0,1]}\mathbb E[|u^{n,\tau}(t_i,\kappa_n(x))|^p]\leq Ce^{Ct_i},$ which implies the intermittent upper bound.

 \textit{Step 2: Intermittency lower bound: $\underline{\gamma}^{n,\tau}(p)>0,\,p\in(1,3).$} Applying \cite[Lemmas 5.4 and 3.4]{Chong} yields that for $p\in(1,2),$
\begin{align*}
&\quad \mathbb E[|u^{n,\tau}(t_i,x)|^p]\\
&\ge CI_0^p+C\mathbb E\Big[\Big|\int_0^{t_i}\int_0^1G^{n,\tau}_2(t_i-\kappa_{\tau}(s)-\tau,x,y)\sigma(u^{n,\tau}(\kappa_{\tau}(s),\kappa_n(y)))\Lambda (ds,dy)\Big|^p\Big]\\
&\ge CI_0^p+C\int_0^{t_i}\int_0^1|G^{n,\tau}_2(t_i-\kappa_{\tau}(s)-\tau,x,y)|^pdy\inf_{y\in[0,1]}\mathbb E[|u^{n,\tau}(\kappa_{\tau}(s),y)|^p]ds.
\end{align*}
Note that a similar proof to \cite[Lemma 4.1 (\romannumeral4)]{dang2022} gives that there is a number $t(n,\tau)>0$ with $\frac{t(n,\tau)}{\tau}$ being an integer, so that $G^{n,\tau}_2(t,x,y)\ge \frac12$ when $t\ge t(n,\tau).$ Thus we have 
$
\mathbb E[|u^{n,\tau}(t_i,x)|^p]\ge CI_0^p+C\tau\sum_{j=0}^{i-\frac{t(n,\tau)}{\tau}-1}\inf_{y\in[0,1]}\mathbb E[|u^{n,\tau}(j\tau,y)|^p],$ 
which together with Lemma \ref{Gronwall} yields 
$\inf_{x\in[0,1]}\mathbb E[|u^{n,\tau}(t_i,x)|^p]\ge  CI_0^p(1+C\tau)^{i-\frac{t(n,\tau)}{\tau}}.$ This implies that the $p$th lower moment Lyapunov exponent of the numerical solution is positive, which completes the proof. 
\end{proof}

\subsection{Preservation of the path property}
In this subsection, we investigate that the numerical solution of the fully discrete scheme preserves the path property \eqref{path_ex2} of the exact solution. 
\begin{theorem}\label{prop3.3}
Let Assumptions    \ref{assum_noise} and \ref{Assumption3} hold and  $\sigma$ be bounded. Then for any $t\in(0,T]$ and $h\in(0,1\wedge t),$ 
\begin{align}\label{cadlag}\sup_{n,\tau}\mathbb E\Big[|\textrm{osc}_r(u^{n,\tau}(t+h),u^{n,\tau}(t))\textrm{osc}_r(u^{n,\tau}(t-h),u^{n,\tau}(t))|^2\Big]\leq Ch^{1+\delta}
\end{align} 
holds for some $\delta>0$ and any  $r<-\frac12.$
\end{theorem}

\begin{proof}
Without loss of generality, we assume that $\tilde b=0$ in \eqref{trun_noise}. First  suppose $u_0\equiv0$, and we consider \begin{align}\label{consi_u}
u^{n,\tau}(t)=\int_0^t\int_0^1G^{n,\tau}_2(t-\kappa_{\tau}(s)-\tau,x,y)\sigma(u^{n,\tau}(\kappa_{\tau}(s),\kappa_n(y)))\Lambda(ds,dy).
\end{align} 
Suppose that $u^{n,\tau}(t)=\sum_{l=0}^{n-1}a^{n,\tau}_{l}(t)f_l,$ where 
\begin{align}\label{a_l}
a^{n,\tau}_{l}(t)&:=(R_{1,l}R_{2,l})^{[\frac {t}{\tau}]}R_{1,l}\sqrt n\int_0^t\int_0^1(R_{1,l}R_{2,l})^{-[\frac{\kappa_{\tau}(s)+\tau}{\tau}]}\mathbf 1_{\{t-\kappa_{\tau}(s)-\tau\ge 0\}}(s)\times\notag\\
&\quad \;\sigma(u^{n,\tau}(\kappa_{\tau}(s),\kappa_n(y)))\bar{e}_l(\kappa_n(y))\Lambda (ds,dy)=:(R_{1,l}R_{2,l})^{[\frac {t}{\tau}]}R_{1,l}\sqrt nI_0^t(l).
\end{align}

To proceed, we establish the relation between the norm in $H^r$ with $r\leq 0$ and its discrete counterpart.  Let $(v(x_0),\ldots,v(x_{n-1}))\in \mathbb C^n$ with $x_i=\frac in,i=0,\ldots,n-1,$   and define the function $v(x):=v(\kappa_n(x)),x\in[0,1].$ Suppose that $\sum_{j=0}^{n-1}(1-\lambda^n_j)^r|\frac1n \sum_{l=0}^{n-1}v(x_l)e^{-2\pi\mathbf ijx_l}|^2<\infty.$ We aim to show that $\bbE[\|v\|^2_{H^r}]<\infty.$ Let the expansion of $v$ in $L^2(0,1)$ be $v=\sum\limits_{j=0}^{\infty}v_je_j,$ where 
$
v_j=\sum\limits_{r=0}^{n-1}v(x_r)e^{-2\pi\mathbf ijx_r}\frac1n\big(\frac{e^{-2\pi\mathbf ijn^{-1}}-1}{-2\pi\mathbf ijn^{-1}}\big)=:\frac{1}{\sqrt n}\tilde v_j\alpha^n_j
$
 with $\alpha^n_j=\frac{e^{-2\pi\mathbf ijn^{-1}}-1}{-2\pi\mathbf ijn^{-1}},\;\tilde v_j=\frac1{\sqrt n}\sum_{r=0}^{n-1}v(x_r)e^{-2\pi\mathbf ijx_r}.$ Then we have 
 \begin{align}
 &\quad \bbE[\|v\|^2_{H^r}]=\sum_{j=0}^{\infty}\frac1n|\tilde v_j|^2|\alpha^n_j|^2(1+4\pi^2j^2)^r\notag\\
 &=\sum_{j=0}^{n-1}\Big(\sum_{l=0}^{\infty}(1+4\pi^2(j+ln)^2)^r|\alpha^n_{j+ln}|^2\Big)|\tilde v_j|^2\frac1n\notag\\
 &\leq \sum_{j=0}^{n-1}(1+4\pi^2j^2)^r\Big(\sum_{l=0}^{\infty}|\alpha^n_{j+ln}|^2\Big)|\tilde v_j|^2\frac1n\leq  \sum_{j=0}^{n-1}(1+4\pi^2j^2)^r\sum_{l=0}^{\infty}(\frac{j}{j+ln})^2|\tilde v_j|^2\frac1n\notag\\
 &\leq C\sum_{j=0}^{n-1}(1-\lambda^n_j)^r|\tilde v_j|^2\frac1n<\infty,\label{sob_rea}
 \end{align}
 where we use the relations $e^{2\pi\mathbf i(n+j)x_r}=e^{2\pi\mathbf ijx_r},r=0,\ldots,n-1$ and $\frac{\alpha^n_{j+n}}{\alpha^n_j}=\frac{j}{j+n}.$ 
  
Noting  that 
 \begin{align*}
 &\quad a^{n,\tau}_{l}(t+h)-a^{n,\tau}_{l}(t)=(R_{1,l}R_{2,l})^{[\frac{t+h}{\tau}]}R_{1,l}\sqrt nI^{t+h}_0(l)-(R_{1,l}R_{2,l})^{[\frac{t}{\tau}]}R_{1,l}\sqrt nI^t_0(l)\\
 &=-(R_{1,l}R_{2,l})^{[\frac{t}{\tau}]}R_{1,l}\sqrt n\Big[\big(1-(R_{1,l}R_{2,l})^{[\frac{t+h}{\tau}]-[\frac{t}{\tau}]}\big)I^t_0(l)-(R_{1,l}R_{2,l})^{[\frac{t+h}{\tau}]-[\frac{t}{\tau}]}I^{t+h}_t(l)\Big],
 \end{align*} and 
 \begin{align*}
 &\quad a^{n,\tau}_{l}(t-h)-a^{n,\tau}_{l}(t)=(R_{1,l}R_{2,l})^{[\frac{t-h}{\tau}]}R_{1,l}\sqrt nI^{t-h}_0(l)-(R_{1,l}R_{2,l})^{[\frac{t}{\tau}]}R_{1,l}\sqrt nI^t_0(l)\\
 &=-(R_{1,l}R_{2,l})^{[\frac{t-h}{\tau}]}R_{1,l}\sqrt n\Big[\big((R_{1,l}R_{2,l})^{[\frac{t}{\tau}]-[\frac{t-h}{\tau}]}-1\big)I^{t-h}_0(l)\\
 &\quad +(R_{1,l}R_{2,l})^{[\frac{t}{\tau}]-[\frac{t-h}{\tau}]}I^t_{t-h}(l)\Big],
 \end{align*}
 we have 
 \begin{align}\label{mathcal Ai}
 &\quad \|u^{n,\tau}(t+h)-u^{n,\tau}(t)\|^2_{H^r}\|u^{n,\tau}(t-h)-u^{n,\tau}(t)\|^2_{H^r}\notag\\
 &\leq C\sum_{l,j=0}^{n-1}(1-\lambda^n_l)^r(1-\lambda^n_j)^r\sum_{i=1}^4|\mathcal A_i(l,j)|^2,
 \end{align}
 where 
 \begin{align*}
 \mathcal A_1(l,j)&=(R_{1,l}R_{2,l})^{[\frac{t}{\tau}]}R_{1,l}(R_{1,j}R_{2,j})^{[\frac{t-h}{\tau}]}R_{1,j}\big(1-(R_{1,l}R_{2,l})^{[\frac{t+h}{\tau}]-[\frac{t}{\tau}]}\big)\times\\
 &\quad \;\big((R_{1,j}R_{2,j})^{[\frac{t}{\tau}]-[\frac{t-h}{\tau}]}-1\big)I^t_0(l)I^{t-h}_0(j),\\
 \mathcal A_{2}(l,j)&=(R_{1,l}R_{2,l})^{[\frac{t}{\tau}]}R_{1,l}(R_{1,j}R_{2,j})^{[\frac{t-h}{\tau}]}R_{1,j}\big(1-(R_{1,l}R_{2,l})^{[\frac{t+h}{\tau}]-[\frac{t}{\tau}]}\big)\times\\
 &\quad \;(R_{1,j}R_{2,j})^{[\frac{t}{\tau}]-[\frac{t-h}{\tau}]}I^t_0(l)I^t_{t-h}(j),\\
 \mathcal A_3(l,j)&=(R_{1,l}R_{2,l})^{[\frac{t}{\tau}]}R_{1,l}(R_{1,j}R_{2,j})^{[\frac{t-h}{\tau}]}R_{1,j}(R_{1,l}R_{2,l})^{[\frac{t+h}{\tau}]-[\frac{t}{\tau}]}\times\\
 &\quad \;\big((R_{1,j}R_{2,j})^{[\frac{t}{\tau}]-[\frac{t-h}{\tau}]}-1\big)I^{t+h}_t(l)I^{t-h}_0(j),\\
 \mathcal A_4(l,j)&=(R_{1,l}R_{2,l})^{[\frac{t}{\tau}]}R_{1,l}(R_{1,j}R_{2,j})^{[\frac{t-h}{\tau}]}R_{1,j}\times\\
 &\quad \;(R_{1,l}R_{2,l})^{[\frac{t+h}{\tau}]-[\frac{t}{\tau}]}(R_{1,j}R_{2,j})^{[\frac{t}{\tau}]-[\frac{t-h}{\tau}]}I^{t+h}_t(l)I^t_{t-h}(j).
 \end{align*}
Below we give the estimates of terms $\mathcal A_i,i=1,2,3,4,$ respectively.  Let   $0\leq t_{i}\leq t<t_{i+1}.$ And without loss of generality, we suppose that $n$ is odd since the even case can be proved similarly.

 \textbf{Estimate of $\mathcal A_1.$} Term $\mathcal A_1$ is further split as $\mathcal A_1=\mathcal A_{1,1}+\mathcal A_{1,2},$ where 
 \begin{align*}
 &\mathcal A_{1,1}(l,j):=(R_{1,l}R_{2,l})^{[\frac{t}{\tau}]}R_{1,l}(R_{1,j}R_{2,j})^{[\frac{t-h}{\tau}]}R_{1,j}\big(1-(R_{1,l}R_{2,l})^{[\frac{t+h}{\tau}]-[\frac{t}{\tau}]}\big)\times \\&\quad\;\big((R_{1,j}R_{2,j})^{[\frac{t}{\tau}]-[\frac{t-h}{\tau}]}-1\big)
I_{t-h}^t(l)I^{t-h}_0(j),\\
&\mathcal A_{1,2}(l,j):= (R_{1,l}R_{2,l})^{[\frac{t}{\tau}]}R_{1,l}(R_{1,j}R_{2,j})^{[\frac{t-h}{\tau}]}R_{1,j}\big(1-(R_{1,l}R_{2,l})^{[\frac{t+h}{\tau}]-[\frac{t}{\tau}]}\big)\times \\
&\quad\; \big((R_{1,j}R_{2,j})^{[\frac{t}{\tau}]-[\frac{t-h}{\tau}]}-1\big)I^{t-h}_0(l)I^{t-h}_0(j).
 \end{align*} 
Recall the definition of the discrete Green function. Noting that $t-h-\kappa_{\tau}(s)-\tau<0$ when $s\in[[\frac{t-h}{\tau}]\tau,t-h),$ the second moment of $\mathcal A_{1,1}$ can be estimated as  \begin{align*}
 &\mathbb E[|\mathcal A_{1,1}(l,j)|^2]\leq C(R_{1,l}R_{2,l})^{2[\frac{t}{\tau}]}(R_{1,l})^2(R_{1,j}R_{2,j})^{2[\frac{t-h}{\tau}]}(R_{1,j})^2\times\\
 &\quad \;\big(1-(R_{1,l}R_{2,l})^{[\frac{t+h}{\tau}]-[\frac{t}{\tau}]}\big)^2\big((R_{1,j}R_{2,j})^{[\frac{t}{\tau}]-[\frac{t-h}{\tau}]}-1\big)^2 \times\\
 &\quad\;\int_0^{[\frac{t-h}{\tau}]\tau}|R_{1,j}R_{2,j}|^{-2[\frac{\kappa_{\tau}(s)+\tau}{\tau}]}ds
 \int_{t-h}^t|R_{1,l}R_{2,l}|^{-2[\frac{\kappa_{\tau}(s)+\tau}{\tau}]}ds.
 \end{align*}
  
\textit{Case 1: $h\ge \tau.$} In this case, $t-h<t_{i}.$  Note that  $R_{1,n-j}=R_{1,j},R_{2,n-j}=R_{2,j},j=1,\ldots,[\frac n2]$ due to $\lambda^n_{n-j}=\lambda^n_j,j=1,\ldots,[\frac n2]. $  
We split the set $\{j:1,2,\ldots,\big[\frac{n}{2}\big]\}$ as
\begin{align}\label{set_split}
\big\{j:1,2,\ldots,\big[\frac{n}{2}\big]\big\}&=\Big\{j:R_{1,j}R_{2,j}\ge \frac{1}{2}\Big\}\cup \Big\{j:-1+\epsilon\leq R_{1,j}R_{2,j}<\frac{1}{2}\Big\}=:A_1\cup A_2.
\end{align}
Denote $R_{3,j}:=(R_{1,j}R_{2,j})^{-1}-1.$ For $j\in A_1,$ $\frac{1}{2}<R_{2,j}<1$ and $-\lambda^n_j\tau\leq R_{3,j}\leq -2\lambda^n_j\tau$, and for $j\in A_2$, $|R_{1,j}R_{2,j}|\leq 1-\epsilon$. Moreover, 
\begin{align}\label{set_prop}A_1\subset \Big\{j:1\leq j\leq \frac{1}{4}\sqrt{\frac{1}{(2-\theta)\tau}}\Big\} \text{ and } A_2\subset \Big\{j:\frac{1}{2\pi}\sqrt{\frac{1}{(2-\theta)\tau}}<j\leq \big[\frac{n}{2}\big]\Big\}.
\end{align}

For  $j\in A_1,$ by the inequality $1-e^{-x}\leq x^{\gamma}$ for $x>0,\gamma\in(0,1),$ we have 
\begin{align}\label{expo_esti}
1-(R_{1,j}R_{2,j})^{[\frac{t}{\tau}]-[\frac{t-h}{\tau}]}= 1-e^{-([\frac{t}{\tau}]-[\frac{t-h}{\tau}])\ln (1+R_{3,j})}\leq C\Big|\Big(\big[\frac{t}{\tau}\big]-\big[\frac{t-h}{\tau}\big]\Big)\lambda^n_j\tau\Big|^{\gamma},
\end{align} 
and for $j\in A_2,$ we have $|R_{1,j}R_{2,j}|\leq 1-\epsilon.$
Hence we obtain that for $j\in A_1,$
\begin{align}
 \mathcal L_0&:= (R_{1,j})^2(1-(R_{1,j}R_{2,j})^{[\frac{t}{\tau}]-[\frac{t-h}{\tau}]})^2\int_0^{[\frac{t-h}{\tau}]\tau}|R_{1,j}R_{2,j}|^{2([\frac{t-h}{\tau}]-[\frac{\kappa_{\tau}(s)+\tau}{\tau}])}ds\notag\\
&\leq C\Big|\Big(\big[\frac{t}{\tau}\big]-\big[\frac{t-h}{\tau}\big]\Big)\lambda^n_j\tau\Big|^{2\gamma}\frac{(R_{1,j})^2\tau}{1-(R_{1,j}R_{2,j})^2}\notag\\
&\leq C\Big|\Big(\big[\frac{t}{\tau}\big]-\big[\frac{t-h}{\tau}\big]\Big)\lambda^n_j\tau\Big|^{2\gamma}(2+(1-2\theta)\tau\lambda^n_j)^{-1}(-\tau\lambda^n_j)^{-1}\tau.\label{mathcal L0}
\end{align}
 When $\theta\in[0,\frac12],$ it holds $2+(1-2\theta)\tau\lambda^n_j\ge 2-\frac{\pi^2(1-2\theta)}{4(2-\theta)}\mathbf 1_{\{\theta\neq \frac12\}}>0$; and when $\theta\in(\frac12,1],$ it holds $2+(1-2\theta)\tau\lambda^n_j\ge 2.$ By taking  $\gamma=\frac{1}{2}$ in \eqref{mathcal L0}, we arrive at  $
 \mathcal L_0\leq C h.
$ For $j\in  A_2,$ 
\begin{align}\label{mathcal L1}
\mathcal J_0\leq C\tau\frac{1-(R_{1,j}R_{2,j})^{2(i-2)}}{1-(R_{1,j}R_{2,j})^2}\leq C\frac{1}{1-(1-\epsilon)^2}h.
\end{align}
Thus we obtain that $\mathbb E[|\mathcal A_{1,1}(l,j)|^2]\leq Ch^2.$

\textit{Case 2: $h< \tau.$} 
 In this case, $t-h\in(t_i,t_{i+1})$ or $t-h\in(t_{i-1},t_i);$ $t+h\in(t_i,t_{i+1})$ or $t+h\in(t_{i+1},t_{i+2}).$
 When $t+h\in(t_i,t_{i+1})$ or $t-h\in(t_i,t_{i+1}),$ we have that $1-(R_{1,l}R_{2,l})^{[\frac{t+h}{\tau}]-[\frac{t}{\tau}]}=0$ or $1-(R_{1,l}R_{2,l})^{[\frac{t}{\tau}]-[\frac{t-h}{\tau}]}=0$; When $t+h\in(t_{i+1},t_{i+2})$ and $t-h\in(t_{i-1},t_{i})$ hold simultaneously, we found that it happens only for $h\in(\frac{\tau}{2},\tau).$ And for this setting, the estimate is similar to $\textit{Case 1.}$ 
 
 Combining \textit{Cases 1-2} gives that $\mathbb E[|\mathcal A_{1,1}(l,j)|]\leq Ch^2.$ Hence, for $r<-\frac12$, we obtain  $\sum_{l,j=0}^{n-1}(1-\lambda^n_l)^r(1-\lambda^n_j)^r\mathbb E[|\mathcal A_{1,1}(l,j)|^2]\leq Ch^{2},\;h\in(0,1).$

 The second moment of $\mathcal A_{1,2}$ can be estimated as
 \begin{align*}
&\quad  \mathbb E[|\mathcal A_{1,2}(l,j)|^2]\leq C(R_{1,l}R_{2,l})^{2[\frac{t}{\tau}]}(R_{1,l})^2(R_{1,j}R_{2,j})^{2[\frac{t-h}{\tau}]}(R_{1,j})^2\times\\
 &\quad \big(1-(R_{1,l}R_{2,l})^{[\frac{t+h}{\tau}]-[\frac{t}{\tau}]}\big)^2\big((R_{1,j}R_{2,j})^{[\frac{t}{\tau}]-[\frac{t-h}{\tau}]}-1\big)^2\times\\
 &\quad \Big[\int_0^{[\frac{t-h}{\tau}]\tau}|R_{1,j}R_{2,j}|^{-4[\frac{\kappa_{\tau}(s)+\tau}{\tau}]}ds\int^{[\frac{t-h}{\tau}]\tau}_0|R_{1,l}R_{2,l}|^{-4[\frac{\kappa_{\tau}(s)+\tau}{\tau}]}ds\Big]^{\frac12}.
 \end{align*}
 
 \textit{Case 1:  $h\ge \tau.$} Similar to \textit{Case 1} in the estimate of $\mathcal A_{1,1},$ we have that for $j,l\in  A_1,$
 \begin{align*}
 &\quad \tilde{\mathcal L}_0:= (R_{1,j})^2(1-(R_{1,j}R_{2,j})^{[\frac{t}{\tau}]-[\frac{t-h}{\tau}]})^2\Big(\int_0^{[\frac{t-h}{\tau}]\tau}|R_{1,j}R_{2,j}|^{4([\frac{t-h}{\tau}]-[\frac{\kappa_{\tau}(s)+\tau}{\tau}])}ds\Big)^{\frac12}\\
 &\leq  C\Big|\Big(\big[\frac{t}{\tau}\big]-\big[\frac{t-h}{\tau}\big]\Big)\lambda^n_j\tau\Big|^{2\gamma} (R_{1,j})^2\Big((1-(R_{1,j}R_{2,j})^2)(1+(R_{1,j}R_{2,j})^2)\tau^{-1}\Big)^{-\frac12}\\
 &\leq C\Big|\Big(\big[\frac{t}{\tau}\big]-\big[\frac{t-h}{\tau}\big]\Big)\lambda^n_j\tau\Big|^{2\gamma} (\lambda^n_j)^{-\frac12}\leq C|\lambda^n_j|^{\delta}h^{\frac12+\delta},
 \end{align*}
 where in the last step we take $\gamma=\frac14+\frac{\delta}{2}$ with some $\delta\in(0,\frac12).$ For $j\in A_2,$ we have $\frac1j\leq C\sqrt{\tau}$ and thus $\tilde {\mathcal L}_0\leq C\sqrt h\leq Ch^{\frac12+\delta}|\lambda^n_j|^{\delta}.$

 \textit{Case 2: $h<\tau.$}
 The analysis in this case is 
 similar to that of $\mathcal A_{1,1}$ and thus is omitted. 
 
 Combining \textit{Cases 1-2} gives that $\mathbb E{|\mathcal A_{1,2}(l,j)|^2}\leq Ch^{1+2\delta}|\lambda^n_j|^{\delta}|\lambda^n_l|^{\delta}.$ 
 Hence, for $r<-\frac12$, there exists some $\delta_0$ such that $\delta<\delta_0,$ we obtain  $\sum_{l,j=0}^{n-1}(1-\lambda^n_l)^r(1-\lambda^n_j)^r\mathbb E[|\mathcal A_{1,2}(l,j)|^2]\leq Ch^{1+2\delta}.$

 \textbf{Estimate of $\mathcal A_2.$} Term $\mathcal A_2$ is further split as 
  $\mathcal A_{2}(l,j)=\mathcal A_{2,1}(l,j)+\mathcal A_{2,2}(l,j),$ where 
\begin{align*}
 &\mathcal A_{2,1}(l,j):=(R_{1,l}R_{2,l})^{[\frac{t}{\tau}]}R_{1,l}(R_{1,j}R_{2,j})^{[\frac{t-h}{\tau}]}R_{1,j}\big(1-(R_{1,l}R_{2,l})^{[\frac{t+h}{\tau}]-[\frac{t}{\tau}]}\big)\times\\
 &\quad\;(R_{1,j}R_{2,j})^{[\frac{t}{\tau}]-[\frac{t-h}{\tau}]}I^{t-h}_0(l)I^t_{t-h}(j),\\
 &\mathcal A_{2,2}(l,j):=(R_{1,l}R_{2,l})^{[\frac{t}{\tau}]}R_{1,l}(R_{1,j}R_{2,j})^{[\frac{t-h}{\tau}]}R_{1,j}\big(1-(R_{1,l}R_{2,l})^{[\frac{t+h}{\tau}]-[\frac{t}{\tau}]}\big)\times\\
 &\quad\; (R_{1,j}R_{2,j})^{[\frac{t}{\tau}]-[\frac{t-h}{\tau}]}I_{t-h}^t(l)I^t_{t-h}(j).
 \end{align*}
 The estimate of  $\mathcal A_{2,1}$ is similar to that of $\mathcal A_{1,1},$ and one can obtain $\mathbb E[|\mathcal A_{2,1}(l,j)|^2]\leq Ch^2.$ The proof is thus omitted. For the term $\mathcal A_{2,2},$ we have 
 \begin{align*}
 &\mathbb E[|\mathcal A_{2,2}(l,j)|^2]\leq C(R_{1,l})^2(R_{1,j})^2\big(1-(R_{1,l}R_{2,l})^{[\frac{t+h}{\tau}]-[\frac{t}{\tau}]}\big)^2\times\\
 &\Big(\int_{t-h}^t|R_{1,j}R_{2,j}|^{4[\frac{t-\kappa_{\tau}(s)-\tau}{\tau}]}ds\Big)^{\frac12}\Big(\int_{t-h}^t|R_{1,l}R_{2,l}|^{4[\frac{t-\kappa_{\tau}(s)-\tau}{\tau}]}\mathbf 1_{\{t-\kappa_{\tau}(s)-\tau\ge 0\}}(s)ds\Big)^{\frac12}.
 \end{align*}

 \textit{Case 1: $h\ge \tau.$} When  $l\in A_1,$ we use  \eqref{expo_esti}, and when $l\in A_2,$ we use $\frac1l\leq C\sqrt{\tau}$ to obtain 
 \begin{align*}
 (R_{1,l})^2\big(1-(R_{1,l}R_{2,l})^{[\frac{t+h}{\tau}]-[\frac{t}{\tau}]}\big)^2\Big(\int_{t-h}^t|R_{1,l}R_{2,l}|^{4[\frac{t-\kappa_{\tau}(s)-\tau}{\tau}]}ds\Big)^{\frac12}\leq C|\lambda^n_l|^{\delta}h^{\frac12+\delta}.
 \end{align*}
 
 \textit{Case 2: $h<\tau$.} For $t+h\in(t_i,t_{i+1}),$ we have $\mathbb E[|\mathcal A_{2,2}(l,j)|^2]=0.$ For $t+h\in(t_{i+1},t_{i+2}),t-h\in(t_i,t_{i+1}),$ by the definition of the discrete Green function, we obtain  $\mathbb E[|\mathcal A_{2,2}(l,j)|^2]=0.$ For $t+h\in(t_{i+1},t_{i+2}),t-h\in(t_{i-1},t_{i}),$ it holds $h\in(\frac{\tau}{2},\tau).$ And in this setting, the estimate is similar to that of \textit{Case 1}. 
 
 Combining \textit{Cases 1-2} leads to that $\mathbb E[|\mathcal A_{2,2}(l,j)|^2]\leq Ch^{1+\delta}|\lambda^n_l|^{\delta}.$ Hence, for $r<-\frac12$, there exists some $\delta_0$ such that $\delta<\delta_0,$ we obtain  $\sum_{l,j=0}^{n-1}(1-\lambda^n_l)^r(1-\lambda^n_j)^r\mathbb E[|\mathcal A_{2,2}(l,j)|^2]\leq Ch^{1+\delta}.$  
 
 \textbf{Estimate of $\mathcal A_3$.} 
Similar to the estimate of $\mathcal A_{1,1}$, we obtain $\mathbb E[|\mathcal A_3(l,j)|^2]\leq Ch^2.$
 
 \textbf{Estimate of $\mathcal A_4$.} It is straightforward that $\mathbb E[|\mathcal A_4(l,j)|^2]
  \leq Ch^2.$ 

Inserting estimates of $\mathcal A_i,i=1,2,3,4$ into \eqref{mathcal Ai} implies \eqref{cadlag} for the case of $u_0\equiv0.$  
 When the initial value $u_0$ is not identically zero, it suffices to estimate the  integral $\mathcal I^{n,\tau}_{0}(t,x):=\int_0^tG^{n,\tau}_1(t,x,y)u_0(\kappa_n(y))dy,t>0,x\in[0,1]$. It follows from the  expression $u_0(\kappa_n(y))=\sum_{j=0}^{n-1}u_{0,j}\bar e_j(\kappa_n(y))$ with $u_{0,j}=\int_0^1u_0(\kappa_n(y))\bar e_j(\kappa_n(y))dy$ 
that  
  \begin{align*}
&\quad  \mathcal I_0:=\|\mathcal I^{n,\tau}_0(t+h,\cdot)-\mathcal I^{n,\tau}_0(t,\cdot)\|^2_{H^r}\|\mathcal I^{n,\tau}_0(t-h,\cdot)-\mathcal I^{n,\tau}_0(t,\cdot)\|^2_{H^r}\\
&\leq C\sum_{j,l=0}^{n-1}(1-\lambda^n_j)^r(1-\lambda^n_l)^r(R_{1,j}R_{2,j})^{2[\frac{t}{\tau}]}(1-(R_{1,j}R_{2,j})^{[\frac{t+h}{\tau}]-[\frac{t}{\tau}]})^2|u_{0,j}|^2\times\\
&\quad (R_{1,l}R_{2,l})^{2[\frac{t-h}{\tau}]}(1-(R_{1,l}R_{2,l})^{[\frac{t}{\tau}]-[\frac{t-h}{\tau}]})^2|u_{0,l}|^2.
 \end{align*}
 When $h\ge \tau,$ we have \begin{align*}
 \mathcal I_0&\leq C\sum_{j,l=0}^{n-1}(1-\lambda^n_j)^r(1-\lambda^n_l)^r\Big|\Big([\frac{t+h}{\tau}]-[\frac{t}{\tau}]\Big)\lambda^n_j\tau\Big|^{2\gamma}\Big|\Big([\frac{t}{\tau}]-[\frac{t-h}{\tau}]\Big)\lambda^n_j\tau\Big|^{2\gamma}\times\\
 & \quad |u_{0,j}|^2|u_{0,l}|^2\leq C\Big(\int_0^1|u_0(\kappa_n(y))|^2dy\Big)^2h^{-2r},
 \end{align*}
 where in the last step we take $\gamma=-r/2.$ The case of $h<\tau$ is similar as before and thus is omitted. The proof is finished.  
 \end{proof}

In fact, the sequence $\{u^{n,\tau}\}_{n,\tau}$ has more fruitful properties. 
 We first show that 
$\sup\limits_{n,\tau}\sup\limits_{t\in[0,T]}\mathbb E[\|u^{n,\tau}(t)\|^2_H]\leq C.$ 
By \eqref{mild full}, we have  \begin{align*}
&\mathbb E[\|u^{n,\tau}(t)\|^2_H]\leq C\int_0^1\Big|\int_0^1G^{n,\tau}_1(t,x,y)u_0(\kappa_n(y))dy\Big|^2dx\\
&+C\int_0^1\mathbb E\Big[\Big|\int_0^t\int_0^1G^{n,\tau}_2(t-\kappa_{\tau}(s)-\tau,x,y)\sigma(u^{n,\tau}(\kappa_{\tau}(s),\kappa_n(y))) \Lambda(ds,dy)\Big|^2\Big]dx\\
&\leq C\sum_{j=0}^{n-1}(R_{1,j}R_{2,j})^{2[\frac{t}{\tau}]}|u_{0,j}|^2+C\int_0^t\Big(1+\frac{1}{\sqrt{[\frac{t-\kappa_{\tau}(s)-\tau}{\tau}]\tau+\tau}}\Big)\mathbf 1_{\{t-\kappa_{\tau}(s)-\tau\ge 0\}}(s)ds\\
&\leq C\int_0^1|u_0(\kappa_n(y))|^2dy+C\int_0^{[\frac{t}{\tau}]\tau}\Big(1+\frac{1}{\sqrt {[\frac{t}{\tau}]\tau-s}}\Big)ds\leq C(1+\sup_{x\in[0,1]}|u_0(x)|).
\end{align*}
Then we show the tightness of $\{u^{n,\tau}\}_{n,\tau}$. 
For each $ \rho>0$ and $t\in[0,T],$ let $\Gamma_{\rho,t}:=\{x\in H:\|x\|_H\leq R(\rho)\}.$ By the compact Sobolev embedding theorem, we have that $\Gamma_{\rho,t}$ is compact in $H^{r}$ with $r<-\frac12$, and 
\begin{align*}
\mathbb P(u^{n,\tau}(t)\in\Gamma_{\rho,t})
\ge 1-\frac{\sup_{n,\tau}\sup_{t\in[0,T]}\mathbb E[\|u^{n,\tau}(t)\|_H]}{R(\rho)}\ge 1-\frac{C}{R(\rho)}.
\end{align*} 
Taking $R(\rho)=C/\rho$ derives the tightness of $\{u^{n,\tau}\}_{n,\tau}$. 
In addition,  for $r<-\frac12,$  \begin{align*}
 &\mathbb E[\|u^{n,\tau}(h)-u^{n,\tau}(0)\|^2_{H^r}]\leq C\sum_{j=0}^{n-1}(1-\lambda^n_j)^r(1-(R_{1,j}R_{2,j})^{[\frac{h}{\tau}]})^{2}\mathbf 1_{\{h>\tau\}}|u_{0,j}|^2\\
 &+C\sum_{j=0}^{n-1}(1-\lambda^n_j)^r(R_{1,j}R_{2,j})^{2[\frac{h}{\tau}]}R^2_{1,j}\int_0^{[\frac{h}{\tau}]\tau}(R_{1,j}R_{2,j})^{-2\frac{\kappa_{\tau}(s)+\tau}{\tau}}ds\leq C(h^{-r}+h)\to0,
 \end{align*}
 as  $h\to0.$ 
Hence, based on \cite[Theorems 8.6 and 8.8]{Kurtz86}, as a corollary of  Theorem \ref{prop3.3}, 
 the numerical solution $\{u^{n,\tau}\}_{n,\tau}$  is weakly relatively compact in the Skorohod space $D([0,T];H^r)$ with $r<-\frac12$.

\section{Convergence  of fully discrete scheme}\label{sec_con}
 This section focuses on the convergence analysis of the fully discrete scheme. 
The analysis  is  based on error estimates between the discrete Green functions and the Green function of the exact solution, whose proofs are postponed to the next section.

\begin{lemma}\label{lemma4.8}
(\romannumeral1) Under Assumption \ref{Assumption3}, 
there is a constant $C>0$ such that for all $x\in[0,1]$,
\begin{align}\label{estimate}
\int_{0}^{\infty}\int_{0}^{1}|G(t,x,y)-G^{n,\tau}_2(t,x,y)|^2\,dydt\leq C\big(\frac{1}{n}+\sqrt{\tau}\big).
\end{align}
(\romannumeral2) Under Assumption \ref{Assumption3} (\romannumeral1)  (\romannumeral2) or $\theta=1$ or $\theta\in (\frac{1}{2},1)$ with $u_0\in H^1$,
for any $\alpha \in(\frac{1}{2}, 1)$, there is a positive constant $C:=C(\alpha)$ such that 
\begin{align}\label{green1}
\Big|\int_0^1\big(G(t,x,y)-G^{n,\tau}_1(t,x,y)\big)u_0(\kappa_n(y))\,dy\Big|^2\leq C\tau^{\alpha-\frac{1}{2}}\big(\big[\frac{t}{\tau}\big]\tau\big)^{-\alpha}+Cn^{1-2\alpha}t^{-\alpha}
\end{align}
for all $x\in[0,1], \;t\ge \tau>0$.
\end{lemma}

With this lemma in hand, we present    the convergence order of the fully  discrete scheme.

\begin{theorem}\label{sec_finite}
Let $m_{\lambda}(2)<\infty$ 
and   
 conditions in Lemma \ref{lemma4.8} hold. Then for each $t\in(0,T]$, there is a constant $C:=C(t)>0$ such that
\begin{align}\label{eqthm4.9}
\sup_{x\in[0,1]}\|u^{n,\tau}(t,x)-u(t,x)\|_{2}\leq C(\tau^{\frac14-}+(\frac1n)^{\frac12-}).
\end{align}
\end{theorem} 
\begin{proof}
We split the proof into two steps. 

\textit{Step 1. Estimate of error $\sup\limits_{x\in[0,1]}\|u(t,x)-u(\kappa_{\tau}(t),\kappa_n(x))\|_{2}.$} By the mild form of the exact solution, we derive  
\begin{align*}
&\bbE[|u(t,x)-u(\kappa_{\tau}(t),x)|^2]\leq C\int_0^1|G(t,x,y)-G(\kappa_{\tau}(t),x,y)|^2|u_0(y)|^2dy\\
&+C\int_{\kappa_{\tau}(t)}^t\int_0^1|G(t-s,x,y)|^2\bbE[|u(s,y)|^2]dsdy\\
&+C\int_0^{\kappa_{\tau}(t)}\int_0^1|G(t-s,x,y)-G(\kappa_{\tau}(t)-s,x,y)|^2\bbE[|u(s,y)|^2]dsdy.
\end{align*}
Let $t\in[t_i,t_{i+1}),i\ge 1.$ Notice that 
\begin{align*}
&\quad \int_0^{t_i}\int_0^1|G(t-s,x,y)-G(t_i-s,x,y)|^2dyds\notag\\
&=\int_0^{t_i}2\sum_{j=1}^{\infty}e^{-8\pi^2j^2(t_i-s)}(e^{-4\pi^2j^2(t-t_i)}-1)^2ds\leq \sum_{j=1}^{\infty}\frac{C}{j^2}(j^4(t-t_i)^2\wedge 1)\leq C|t-t_i|^{\frac12}, 
\end{align*}
and that for $\alpha_0\in(\frac12,2),$
\begin{align*}
&\int_0^1|G(t,x,y)-G(t_i,x,y)|^2dy\leq C
\sum_{j=1}^{\infty}e^{-8\pi^2j^2t_i}(e^{-4\pi^2j^2(t-t_i)}-1)^2\leq C\tau^{\alpha_0-\frac12}t_i^{-\alpha_0}. 
\end{align*}
These, combining  \cite[Lemma 2.1]{dang2022} lead to $$
\sup_{x\in[0,1]}\|u(t,x)-u(\kappa_{\tau}(t),x)\|_2\leq C(\tau^{\frac{\alpha_0}{2}-\frac14}(\kappa_{\tau}(t))^{-\frac{\alpha_0}{2}}+\tau^{\frac14}),\quad  t\ge \tau.
$$
For $t\in(0,\tau),$ it is sufficient  to estimate  the integral with the initial value $u_0(x)=\sum_{k=0}^{\infty}u_{0,k}e_k(x).$ By the Sobolev embedding theorem, we obtain   \begin{align*}
&\quad \sup_{x\in[0,1]}\Big|\int_0^1G(t,x,y)u_0(y)dy-u_0(x)\Big|^2\leq \Big\|\sum_{k=0}^{\infty}(1-e^{-4\pi^2k^2t})u_{0,k}e_k\Big\|^2_{H^{\frac12+}}\\
&\leq \sum_{k=0}^{\infty}(1+4\pi^2k^2)^{\frac12+}(1-e^{-4\pi^2k^2t})^2|u_{0,k}|^2\leq Ct^{\frac12-}\|u_0\|^2_{H^1},
\end{align*}
where we use $1-e^{-x}\leq x^{\alpha},\alpha\in(0,1),x>0.$
Hence, we have that for $\alpha_0\in(\frac12,2)$,
\begin{align}\label{time_H1}
\sup_{x\in[0,1]}\|u(t,x)-u(\kappa_{\tau}(t),x)\|_2\leq C(\tau^{\frac{\alpha_0}{2}-\frac14}(\kappa_{\tau}(t))^{-\frac{\alpha_0}{2}}\mathbf 1_{\{t\ge \tau\}}+\tau^{\frac14-}),\quad t>0.
\end{align} 
The $L^2(\Omega)$-error in  space is estimated as
\begin{align*}
&\bbE[|u(t,x)-u(t,\kappa_n(x))|^2]\leq C\int_0^1|G(t,x,y)-G(t,\kappa_n(x),y)|^2|u_0(y)|^2dy\\
&+C\int_0^{t}\int_0^1|G(t-s,x,y)-G(t-s,\kappa_n(x),y)|^2\bbE[|u(s,y)|^2]dsdy.
\end{align*}
It follows from the definition of  Green function $G$  that 
\begin{align*}
&\quad \int_0^t\int_0^1|G(t-s,x,y)-G(t-s,\kappa_n(x),y)|^2dsdy\\
&=2\sum_{j=1}^{\infty}\int_0^te^{-8\pi^2j^2(t-s)}ds|e^{2\pi\mathbf ij(x-\kappa_{n}(x))}-1|^2\leq \sum_{j=1}^{\infty}\frac{C}{j^2}(1\wedge j^2(x-\kappa_n(x))^2)\leq |\frac Cn|^{1-},
\end{align*}
and that for any $\delta\ll1,$
\begin{align*}
&\quad \int_0^1|G(t,x,y)-G(t,\kappa_n(x),y)|^2dy\\
&\leq C\sum_{j=1}^{\infty}e^{-8\pi^2j^2(t-s)}(1\wedge j^2(x-\kappa_n(x))^2)\leq C(\frac{1}{t})^{1-\delta}(\frac1n)^{1-\delta}.
\end{align*}
Hence, we obtain  $\sup\limits_{x\in[0,1]}\|u(t,x)-u(t,\kappa_n(x))\|_2\leq C\big((\frac{1}{t})^{\frac{1-\delta}{2}}+1\big)(\frac1n)^{\frac12-\delta}.$

\textit{Step 2. Estimate of error $\sup\limits_{x_j\in[0,1]}\|u(t_i,x_j)-u^{n,\tau}(t_i,x_j)\|_2,t_i\ge \tau$.}   
From the expression of $u$ and $u^{n,\tau},$ we have
that for $\alpha\in(\frac12,2),$
\begin{align*}
&\quad \mathbb E[|u(t_i,x_j)-u^{n,\tau}(t_i,x_j)|^2]
\leq C\tau^{\alpha-\frac12}t_i^{-\alpha}	\\&+C\int_0^{t_i}\int_0^1(G(t_i-s,x_j,y)-G^{n,\tau}_2(t_i-s,x_j,y))^2dsdy\sup_{s\in[0,t_i]}\sup_{y\in[0,1]}\mathbb E[|u(s,\kappa_n(y))|^2]\\
&+C\int_0^{t_i}\int_0^1(G^{n,\tau}_2(t_i-s,x_j,y))^2\sup_{y\in[0,1]}\mathbb E[|u(s,\kappa_n(y))-u(\kappa_{\tau}(s),\kappa_n(y))|^2]dsdy\\
&+C\int_0^{t_i}\int_0^1(G^{n,\tau}_2(t_i-s,x_j,y)-G^{n,\tau}_2(t_i-\kappa_{\tau}(s)-\tau,x_j,y))^2dsdy\times\\
&\sup_{s\in[0,t_i]}\sup_{y\in[0,1]}\mathbb E[|u(\kappa_{\tau}(s),\kappa_n(y))|^2]+C\int_0^{t_i}\int_0^1(G^{n,\tau}_2(t_i-\kappa_{\tau}(s)-\tau,x_j,y))^2\times\\
&\;\,\mathbb E[|u(\kappa_{\tau}(s),\kappa_n(y))-u^{n,\tau}(\kappa_{\tau}(s),\kappa_n(y))|^2]dsdy.
\end{align*}
To proceed, we need to show   
 \begin{align}\label{show1}
 \int_0^{t_i}\int_0^1(G^{n,\tau}_2(t_i-s,x,y)-G^{n,\tau}_2(t_i-\kappa_{\tau}(s)-\tau,x,y))^2dyds\leq C\sqrt \tau. 
 \end{align}
In fact, similar to \eqref{mathcal L0} and \eqref{mathcal L1}, we have the  estimate 
 \begin{align*}
 &\int_0^{t_i}\sum_{j\in  A_1}(R_{1,j}R_{2,j})^{2[\frac{t_i-\kappa_{\tau}(s)-\tau}{\tau}]}((R_{1,j}R_{2,j})^{[\frac{t_i-s}{\tau}]-[\frac{t_i-\kappa_{\tau}(s)-\tau}{\tau}]}-1)^2R_{1,j}^2ds\\
  &\leq \sum_{j\in  A_1}\frac{R_{1,j}^2\tau}{1-(R_{1,j}R_{2,j})^2}\Big|\big([\frac{t_i-s}{\tau}]-[\frac{t_i-\kappa_{\tau}(s)-\tau}{\tau}]\big)\lambda^n_j\tau\Big|
\leq C\sqrt{\tau},
 \end{align*}
  and the estimate 
 \begin{align*}
 &\int_0^{t_i}\sum_{j\in A_2}(R_{1,j}R_{2,j})^{2[\frac{t_i-\kappa_{\tau}(s)-\tau}{\tau}]}((R_{1,j}R_{2,j})^{[\frac{t_i-s}{\tau}]-[\frac{t_i-\kappa_{\tau}(s)-\tau}{\tau}]}-1)^2R_{1,j}^2ds\\
 &\leq C\sum_{j\in A_2}\frac{R_{1,j}^2\tau}{1-(1-\epsilon)^2} \leq C\sum_{j\in A_2}(1+16\theta j^2\tau)^{-2}\tau\leq C\sqrt\tau, 
 \end{align*}
 which imply \eqref{show1}.

Then taking $\alpha_0=1-\delta$ with $\delta$ being small in \eqref{time_H1}, and
 applying   Proposition \ref{stop_solution} and Lemma \ref{lemma4.8}, we derive
\begin{align*}
&\sup_{x_j\in[0,1]}\mathbb E[|u(t_i,x_j)-u^{n,\tau}(t_i,x_j)|^2]\leq C\tau^{\alpha-\frac12}t_i^{-\alpha}+C\sqrt{\tau}\\
&+C\tau^{\frac{1-\delta}{2}}\int_0^{t_i}(1+\frac{1}{\sqrt{t_i-s}})(1+\kappa_{\tau}(s)^{-1+2\epsilon}\mathbf 1_{\{s\ge \tau\}})ds\\
&+C\int_0^{t_i}(1+\frac{1}{\sqrt{t_i-\kappa_{\tau}(s)}})\sup_{y\in[0,1]}\mathbb E[|u(\kappa_{\tau}(s),\kappa_n(y))-u^{n,\tau}(\kappa_{\tau}(s),\kappa_n(y))|^2]ds,
\end{align*}
which together with the Gr\"onwall inequality finishes the proof. 
\end{proof}

At the end of this section, we give some discussions on Assumption \ref{assum_noise} of the noise.  
When $m_{\lambda}(p)$ is not necessarily finite for some $p\ge 1$, the exact solution is still well-posed in the sense that  $\sup_{x\in[0,1]}\bbE[|u(t,x)|^p\mathbf 1_{\{t\leq \tilde{\tau}_N\}}]<\infty$ for $p\in[1,3)$ and $N\in\bbN_+,$   where the stopping time is defined as 
$
\tilde{\tau}_N:=\inf\{t\in[0,T]:\mu([0,t]\times[0,1]\times[-N,N]^c)> 0\},\; N\in\mathbb N_+.
$
In this setting, we can 
introduce a noise truncation skill to obtain a convergent numerical method. Precisely, we truncate the noise \eqref{Lambda} as
\begin{align}\label{trun_noise}
\Lambda_N(dt,dx):&=bdtdx+\int_{|z|\leq 1}z\tilde {\mu}(dt,dx,dz)+\int_{1<|z|\leq N}z\mu(dt,dx,dz)\notag\\
&=\tilde bdtdx+\int_{|z|\leq N}z\tilde {\mu}(dt,dx,dz),
\end{align}
where $\tilde b=b+\int_{1<|z|\leq N}z\nu(dt,dx,dz).$ Denote by $u^{n,\tau}_N(t,x) $ the numerical   solution of the fully discrete scheme \eqref{fullscheme} with the truncated noise \eqref{trun_noise}.  
Then for each fixed $N,$ Assumption \ref{assum_noise} is satisfied, 
and thus the truncated numerical solution $u^{n,\tau}_N$ possesses     properties concerned in this paper.  
Below we present  the almost sure convergence of the truncated numerical solution.  

\begin{corollary}\label{semi_con}
Let 
conditions in Lemma \ref{lemma4.8}  hold, and let  
$n=n_m,\tau=\tau_m$ be sequences of integers such that  $n_m\ge m^{1+\delta},\tau_m\ge m^{-2(1+\delta)}$ for all $m\in\mathbb N_+$ and for  some $\delta>0$. Then for each $(t,x)\in(0,T]\times[0,1],$  $$\lim_{N\to\infty}\lim_{m\to\infty}|u^{n_m,\tau_m}_N(t,x)-u(t,x)|=0\quad a.s.
$$ 
\end{corollary}
\begin{proof}
We denote by $u_N(t,x)$ the solution of \eqref{she} with the  truncated noise $\Lambda_N$. 
For the numerical solution with the truncated noise, according to Theorem    \ref{sec_finite}, we have that for each $t\in(0,T],$ $\sup_{x\in[0,1]}\|u^{n_m,\tau_m}_N(t,x)-u_N(t,x)\|_2\leq C(\tau^{\frac14-}+(\frac1n)^{\frac12-}).$ 
Hence, the Chebyshev inequality gives that for each $(t,x)\in(0,T]\times[0,1],$
\begin{align*}
\mathbb P(|u^{n_m,\tau_m}_N(t,x)-u_N(t,x)|>m^{-\gamma})
\leq Cm^{2\gamma}(\tau_m^{\frac12-}+(\frac{1}{n_m})^{1-}).
\end{align*}
By the assumption $n_m\ge m^{1+\delta},\tau_m\ge m^{-2(1+\delta)}$ and the Borell--Cantelli lemma, we have $\lim\limits_{m\to\infty}|u^{n_m,\tau_m}_N(t,x)-u_N(t,x)|=0$ a.s. 

Next, we show that $\lim\limits_{N\to\infty}|u_N(t,x)-u(t,x)|=0$ a.s. for each $(t,x)\in(0,T]\times[0,1].$ 
From the definition of the stopping time $\tilde\tau_N,$ we have  $\Lambda=\Lambda_N$ and $u(t,x)=u_N(t,x)$ on the event $\{T\leq \tilde{\tau}_N\}.$ Noticing  that there is only a finite number of jumps larger than $N$ in the bounded closed set $[0,T]\times[0,1]$, we deduce that the non-decreasing sequence of stopping times $\{\tilde{\tau}_N\}_{N\ge 1}$ satisfies $\tilde{\tau}_N=\infty$ for large values of $N.$ This implies that,  for sufficiently large $N,$ $u_N(t,x)=u(t,x),$ a.s. for $(t,x)\in(0,T]\times[0,1].$ Thus we finish the proof. 
\end{proof}

\section{Estimates of discrete Green functions}\label{sec5}
In this section, we give the  proof of estimates of discrete Green functions in Lemma \ref{lemma4.8}.

\begin{proof}[Proof of Lemma \ref{lemma4.8}]
$(\romannumeral1)$  By expanding the real and imaginary parts, the fully discrete Green functions can be written as follows:
\begin{align*}
G^{n,\tau}_1(t,x,y)
&=1+2\tilde{\sum}_{l}(R_{1,l}R_{2,l})^{[\frac{t}{\tau}]}\big(\varphi_{c,l}(\kappa_n(x))\varphi_{c,l}(\kappa_n(y))\\
&\quad +\varphi_{s,l}(\kappa_n(x))\varphi_{s,l}(\kappa_n(y))\big)+(R_{1,\frac{n}{2}}R_{2,\frac{n}{2}})^{[\frac{t}{\tau}]}g_n(x,y),\\
G^{n,\tau}_2(t,x,y)&=1+2\tilde{\sum}_{l}(R_{1,l}R_{2,l})^{[\frac{t}{\tau}]}R_{1,l}\big(\varphi_{c,l}(\kappa_n(x))\varphi_{c,l}(\kappa_n(y))\\
&\quad +\varphi_{s,l}(\kappa_n(x))\varphi_{s,l}(\kappa_n(y))\big)+(R_{1,\frac{n}{2}}R_{2,\frac{n}{2}})^{[\frac{t}{\tau}]}R_{1,\frac{n}{2}}g_n(x,y),
\end{align*}
where  we use the notation $\tilde{\sum}_{l}$ to denote $\sum_{l=1}^{[\frac n2]}$ when $n$ is odd and to denote $\sum_{l=1}^{\frac n2-1}$ when $n$ is even. Here, $\varphi_{c,l}(x):=\cos(2\pi lx),\varphi_{s,l}(x):=\sin(2\pi lx)$, 
  $
g_n(x,y)=0$ when  $n$ is odd, and $g_n(x,y)=\varphi_{c,\frac{n}{2}}(\kappa_n(x))\varphi_{c,\frac{n}{2}}(\kappa_n(y))$ when $n$ is even. 
Rewrite the spectral decomposition of $G(t,x,y)$ 
as 
$
G(t,x,y)=1+2\sum_{l=1}^{\infty}e^{-4\pi^2l^2t}(\varphi_{c,l}(x)\varphi_{c,l}(y)+\varphi_{s,l}(x)\varphi_{s,l}(y)).
$

We first show the result by supposing $n$ is odd. It is clear that
\begin{align}\label{split_I1}
I:=\int_{0}^{\infty}\int_{0}^{1}|G(t,x,y)-G^{n,\tau}_2(t,x,y)|^2\,dy\,dt\leq 8\sum_{k=1}^{4}(I^c_k+I^s_k),
\end{align}
where
\begin{align*}
&I^c_1:=\int_{0}^{\infty}\int_{0}^{1}\Big|2\sum_{r=[\frac{n}{2}]+1}^{\infty}e^{-4\pi^2r^2t}\varphi_{c,r}(x)\varphi_{c,r}(y)\Big|^2\,dy\,dt,\\
&I^c_2:=\int_{0}^{\infty}\int_{0}^{1}\Big|2\sum_{r=1}^{[\frac{n}{2}]}e^{-4\pi^2r^2t}(\varphi_{c,r}(x)-\varphi_{c,r}(\kappa_n(x)))\varphi_{c,r}(y)\Big|^2\,dy\,dt,\\
&I^c_3:=\int_{0}^{\infty}\int_{0}^{1}\Big|2\sum_{r=1}^{[\frac{n}{2}]}e^{-4\pi^2r^2t}\varphi_{c,r}(\kappa_n(x))(\varphi_{c,r}(y)-\varphi_{c,r}(\kappa_n(y)))\Big|^2\,dy\,dt,\\
&I^c_4:=\int_{0}^{\infty}\int_{0}^{1}\Big|2\sum_{r=1}^{[\frac{n}{2}]}(e^{-4\pi^2r^2t}-(R_{1,r}R_{2,r})^{[\frac{t}{\tau}]}R_{1,r})\varphi_{c,r}(\kappa_n(x))\varphi_{c,r}(\kappa_n(y))\Big|^2\,dy\,dt.
\end{align*} Terms  $I^s_k,k=1,2,3,4$ are defined in a similar way via replacing $\cos(\cdot)$ by $\sin(\cdot)$.

Terms $I^c_1,I^c_2$ can be estimated as follows: when $n\ge 3$,
\begin{align*}
I^c_1&=4\int_{0}^{\infty}\int_{0}^{1}\sum_{r=[\frac{n}{2}]+1}^{\infty}e^{-8\pi^2r^2t}\cos^2(2\pi rx)\cos^2(2\pi ry)\,dy\,dt\\
&\leq C\int_{0}^{\infty}\sum_{r=[\frac{n}{2}]+1}^{\infty}e^{-8\pi^2r^2t}\,dt
\leq \frac{C}{n},
\end{align*} and 
$
I^c_2
\leq C\int_{0}^{\infty}\sum_{r=1}^{[\frac{n}{2}]}e^{-8\pi^2 r^2 t}\times (\frac{2\pi r}{n})^2\,dt\leq \frac{C}{n}.
$ 
Define the notation $\tilde{G}_n(t,x,y):=1+2\sum_{r=1}^{[\frac{n}{2}]}e^{-4\pi^2r^2t}\varphi_{c,r}(\kappa_n(x))\varphi_{c,r}(y)$. 
Note that for every function $v\in \mathcal{C}^1([0,1])$,
\begin{align*}
&\int_{0}^{1}|v(y)-v(\kappa_n(y))|^2\,dy=\int_{0}^{1}\Big|\int_{\kappa_n(y)}^{y}v'(x)\,dx\Big|^2\,dy\leq \frac{1}{n}\int_{0}^{1}\int_{\kappa_n(y)}^{y}|v'(x)|^2\,dx\,dy\\
\leq&\; \frac{1}{n}\int_{0}^{1}\int_{x}^{x+\frac{1}{n}}|v'(x)|^2\,dy\,dx\leq \frac{1}{n^2}\int_{0}^{1}|v'(x)|^2\,dx.
\end{align*}
Thus,
\begin{align*}
I^c_3&=\int_{0}^{\infty}\int_{0}^{1}|\tilde{G}_n(t,x,y)-\tilde{G}_n(t,x,\kappa_n(y))|^2\,dy\,dt\leq \frac{1}{n^2}\int_{0}^{\infty}\int_{0}^{1}\Big|\frac{d}{dy}\tilde{G}_n(t,x,y)\Big|^2\,dy\,dt\\
&\leq \frac{C}{n^2}\int_{0}^{\infty}\int_{0}^{1}\sum_{r=1}^{[\frac{n}{2}]}e^{-8\pi^2r^2t}\times r^2\sin^2(2\pi ry)\,dy\,dt
\leq \frac{C}{n}. 
\end{align*}
Similarly, we can obtain  $I^s_k\leq \frac{C}{n},k=1,2,3$.

For the term $I^c_4,$ we have 
$
I^c_4\leq 2\int_0^{\infty}\sum_{r=1}^{[\frac n2]}|e^{-4\pi^2 r^2t}-(R_{1,r}R_{2,r})^{[\frac{t}{\tau}]}R_{1,r}|^2dt.
$
Recalling  \eqref{set_split}, we have the split  $\{j:1,2,\ldots,[\frac n2]\}=A_1\cup A_2$, and sets $A_1,A_2$ have the property stated  in \eqref{set_prop}. 
Hence,  we have the decomposition of the term $I^c_4,$
\begin{align}\label{split_I2}
&I^c_4
\leq C\sum_{j\in A_1}\int_0^{\infty}|e^{-4\pi^2j^2t}-e^{\lambda^n_jt}|^2dt+C
\sum_{j\in A_1}\int_{0}^{\infty}\Big|e^{\lambda^n_j t}-\exp\Big\{-R_{3,j}\frac{t}{\tau}\Big\}\Big|^2\,dt\notag\\
&+C\sum_{j\in A_1}\int_{0}^{\infty}\Big|\exp\Big\{-R_{3,j}\frac{t}{\tau}\Big\}-\exp\Big\{-R_{3,j}\big[\frac{t}{\tau}\big]\Big\}\Big|^2\,dt\notag\\
&+C\sum_{j\in A_1}\int_{0}^{\infty}\Big|\exp\Big\{-R_{3,j}\big[\frac{t}{\tau}\big]\Big\}-\exp\Big\{-R_{3,j}\big[\frac{t}{\tau}\big]\Big\}R_{1,j}\Big|^2\,dt\notag\\
&+C\sum_{j\in A_1}\int_{0}^{\infty}\Big|\exp\Big\{-R_{3,j}\big[\frac{t}{\tau}\big]\Big\}R_{1,j}-(R_{1,j}R_{2,j})^{[\frac{t}{\tau}]}R_{1,j}\Big|^2\,dt\notag\\
&+C\sum_{j\in A_2}\int_0^{\infty}e^{-8\pi^2j^2t}dt+C\sum_{j\in A_2}\int_0^{\infty}|(R_{1,j}R_{2,j})^{[\frac{t}{\tau}]}R_{1,j}|^2dt
=:C\sum_{i=0}^6\mathcal J_i.
\end{align}
The term $\mathcal J_0$ is estimated as  
\begin{align*}
\mathcal J_0&\leq C\int_{0}^{\infty}\sum_{r=1}^{[\frac{n}{2}]}e^{-8\pi^2 r^2c^n_rt}(1-e^{-4\pi^2r^2(1-c^n_r)t})^2\,dt\\
&\leq C\int_{0}^{\infty}\sum_{r=1}^{[\frac{n}{2}]}e^{-Cr^2t} (4\pi^2r^2(1-c^n_r)t)^2\,dt\leq \frac{C}{n^4}\int_{0}^{\infty}\sum_{r=1}^{[\frac{n}{2}]}r^8t^2e^{-Cr^2t}dt
\leq \frac{C}{n},
\end{align*}
where $c^n_r:=\sin^2\frac{r\pi}{n}/(\frac{r\pi}{n})^2\in [\frac{4}{\pi^2},1]$ for $r=1,2,\ldots,[\frac{n}{2}]$. Here we use the inequalities $1-e^{-z}\leq z$ and $1-\frac{\sin^2 z}{z^2}\leq \frac{z^2}{3}$ for $z>0$. The term $\mathcal J_1$ is estimated as 
\begin{align*}
\mathcal J_1
=&\,\sum_{j\in A_1}\int_{0}^{\infty}e^{2\lambda^n_j t}\times \Big|1-\exp\Big\{-\frac{(1-\theta)(\tau \lambda^n_j)^2}{1+(1-\theta)\tau\lambda^n_j}\frac{t}{\tau}\Big\}\Big|^2\,dt\\
\leq&\,\sum_{1\leq j\leq \frac{1}{4}\sqrt{\frac{1}{(2-\theta)\tau}}}C\int_{0}^{\infty}e^{2\lambda^n_j t}\times (\frac{t}{\tau})^2(\lambda^n_j \tau)^4\,dt
\leq \sum_{1\leq j\leq \frac{1}{4}\sqrt{\frac{1}{(2-\theta)\tau}}}C\tau^2j^2\leq C\sqrt{\tau}. 
\end{align*}
For the term $\mathcal J_2,$ by applying the mean value theorem, we have  
\begin{align*}
\mathcal J_2\leq& \sum_{j\in A_1}\int_{0}^{\infty}\exp\Big\{-2R_{3,j}\big[\frac{t}{\tau}\big]\Big\}|R_{3,j}|^2\,dt
\leq \sum_{j\in A_1}C\int_{0}^{\infty}j^4\tau^2e^{-32j^2\tau [\frac{t}{\tau}]}\,dt
\leq C\sqrt{\tau}.
\end{align*}
For terms $\mathcal J_3$ and $\mathcal J_4$, we obtain  
\begin{align*}
\mathcal J_3\leq& \sum_{j\in A_1}\int_{0}^{\infty}\exp\Big\{-2R_{3,j}\big[\frac{t}{\tau}\big]\Big\}\Big|\frac{-\theta\tau\lambda^n_j}{1-\theta \tau \lambda^n_j}\Big|^2\,dt
\leq \sum_{j\in A_1}C\int_{0}^{\infty}j^4\tau^2e^{-32j^2t}\,dt
\leq C\sqrt{\tau},
\end{align*} and 
\begin{align*}
\mathcal J_4\leq &\sum_{j\in A_1}\int_{0}^{\infty}\exp\Big\{-2\big[\frac{t}{\tau}\big]\ln (1+R_{3,j})\Big\}\\
&\qquad\qquad\times\Big|1-\exp\Big\{\big[\frac{t}{\tau}\big](-R_{3,j}+\ln (1+R_{3,j}))\Big\}\Big|^2\Big(\frac{1}{1-\theta \tau \lambda^n_j}\Big)^2\,dt\\
\leq &\sum_{1\leq j\leq\frac{1}{4}\sqrt{\frac{1}{(2-\theta)\tau}}}\int_{0}^{\infty}\exp\Big\{2C_2\lambda^n_j \tau \big[\frac{t}{\tau}\big]\Big\}\times\Big|1-\exp\Big\{-\big[\frac{t}{\tau}\big]C_1(-2\lambda^n_j\tau)^2\Big\}\Big|^2\,dt\\
\leq &\sum_{1\leq j\leq\frac{1}{4}\sqrt{\frac{1}{(2-\theta)\tau}}}C\int_{0}^{\infty}e^{-32C_2j^2t}\times t^2j^8\tau^2\,dt\leq C\tau^2 \sum_{1\leq j\leq\frac{1}{4}\sqrt{\frac{1}{(2-\theta)\tau}}}j^2\leq C\sqrt{\tau},
\end{align*}
where we use the fact that $z:=-\lambda^n_j\tau\in (0,\frac{\pi^2}{4(2-\theta)}]$, $j=1,2,\ldots\,[\frac{n}{2}]$ because of $j^2\tau\leq \frac{1}{16(2-\theta)}$. For such bounded  $z$, we have $-C_1z^2\leq-z+\ln (1+z)\leq 0$ and $\ln (1+z)\ge C_2z$ for some $C_1,C_2>0$. 
For the term $\mathcal J_5$, we have 
$
\mathcal J_5\leq \sum_{\frac{1}{2\pi}\sqrt{\frac{1}{(2-\theta)\tau}}<j\leq [\frac{n}{2}]}\frac{C}{j^2}\leq C\sqrt{\tau}.
$
And the term $\mathcal J_6$ is estimated as 
\begin{align*}
&\mathcal J_6\leq C\sum_{j\in A_2}\int_{0}^{\infty}(1-\epsilon)^{2[\frac{t}{\tau}]}\times (1+16\theta j^2\tau)^{-2}\,dt\\
\leq  &\;\sqrt{\tau}\int_{\frac{1}{2\pi}\sqrt{\frac{1}{2-\theta}}}^{[\frac{n}{2}]\sqrt{\tau}}\int_0^{\infty}(1-\epsilon)^{2[r]}\times (1+16\theta y^2)^{-2}\,dr\,dy\leq C\sqrt{\tau},
\end{align*}
where we use the condition  $n^2\tau\leq C$ for $\theta=0.$

When $n$ is even, what we need to prove is the difference of the term of $j=\frac{n}{2}$ in the expansions of $G$ and $G^{n,\tau}_2$, i.e.,
\begin{align*}
&\int_{0}^{\infty}\int_{0}^{1}\Big|\Big(e^{-\pi^2n^2t}-(R_{1,\frac{n}{2}}R_{2,\frac{n}{2}})^{[\frac{t}{\tau}]}R_{1,\frac{n}{2}}\Big)g_n(x,y)\Big|^2\,dy\,dt\\
\leq&\; 2\int_{0}^{\infty}e^{-2\pi^2n^2t}\,dt+2\int_{0}^{\infty}(1-\epsilon)^{2[\frac{t}{\tau}]}(1+4\theta n^2\tau)^{-2}\,dt\leq\frac{C}{n^2}+C\tau.
\end{align*}
Hence the proof of $(\romannumeral1)$ is completed.

$(\romannumeral2)$
We suppose that $n$ is odd since the case of $n$ being even  can be proved similarly. We split the proof into two cases. 
  
\textit{Case $1:$ $\theta\in[0,\frac{1}{2}]$ or $\theta=1.$} In this case, we prove that for any $\alpha\in(\frac{1}{2},1)$,
\begin{align}\label{prove}
\int_{0}^{1}|G(t,x,y)-G^{n,\tau}_1(t,x,y)|^2\,dy\leq C\tau^{\alpha-\frac{1}{2}}\big(\big[\frac{t}{\tau}\big]\tau\big)^{-\alpha}+Cn^{1-2\alpha}t^{-\alpha}
\end{align}
with some $C:=C(\alpha)>0$ for all $x\in[0,1],t\ge \tau$.
Similar to \eqref{split_I1}, the estimate of the left-hand side of \eqref{prove} can be divided into eight subterms. The difference is to replace  $G^{n,\tau}_2$ by $G^{n,\tau}_1$ and to remove  the integral with $t$, and we still  denote these subterms    by $I^c_j,j=1,2,3,4,$ and $I^s_j,j=1,2,3,4.$ Below we only show the estimates of subterms 
 $I^c_j,j=1,2,3,4.$

Using  the inequality $e^{-z}\leq C(\alpha)z^{-\alpha}$ for $z>0,\alpha >0$, we have that for $\alpha>\frac12,$
\begin{align*}
I^c_1\leq C\sum_{r=[\frac{n}{2}]+1}^{\infty}e^{-8\pi^2 r^2t}\leq C\sum_{r=[\frac{n}{2}]+1}^{\infty}r^{-2\alpha}t^{-\alpha}\leq C\int_{[\frac{n}{2}]}^{\infty}x^{-2\alpha}\,dx\times t^{-\alpha}\leq C n^{1-2\alpha}t^{-\alpha},
\end{align*}
and that for $\alpha<\frac32,$
\begin{align*}
I^c_2+I^c_3&\leq C\sum_{r=1}^{[\frac{n}{2}]}r^2n^{-2}e^{-8\pi^2 r^2t}\leq C\sum_{r=1}^{[\frac{n}{2}]}r^{2-2\alpha}t^{-\alpha}n^{-2}
\leq Cn^{1-2\alpha}t^{-\alpha}.
\end{align*}
Similar to \eqref{split_I2}, here we still divide the estimate of the term $I^c_4$ into estimates of terms $\mathcal J_i,i=0,\ldots,6.$  
The term $\mathcal J_0$ is estimated  as 
\begin{align*}
\mathcal J_0\leq C\sum_{r=1}^{[\frac{n}{2}]}e^{-Cr^2t}\times t^2r^8n^{-4}\leq C\sum_{r=1}^{[\frac{n}{2}]}r^{8-2\gamma}t^{2-\gamma}n^{-4}\leq Cn^{5-2\gamma}t^{2-\gamma}=:Cn^{1-2\alpha}t^{-\alpha}
\end{align*}
for $\alpha<\frac 52$. 
 When $t\ge \tau$, we have \begin{align*}
\mathcal J_1=&\sum_{j\in A_1}e^{2\lambda^n_j t}\times \Big|1-\exp\Big\{\lambda^n_j \tau\Big(\frac{1}{1+(1-\theta)\tau \lambda^n_j}-1\Big)\frac{t}{\tau}\Big\}\Big|^2\\
\leq& \sum_{j\in A_1}Ce^{2\lambda^n_j t}\times j^8\tau^4(\frac{t}{\tau})^2
\leq C\sum_{1\leq j\leq \frac{1}{4}\sqrt{\frac{1}{(2-\theta)\tau}}}(\frac{t}{\tau})^{2-\gamma}(j^2\tau)^{4-\gamma}
\leq C\tau^{\alpha-\frac{1}{2}}t^{-\alpha},
\end{align*}
where we let $\alpha=\gamma-2$ in the last step, 
and for $\alpha<2$,
\begin{align*}
\mathcal J_2=&\sum_{j\in A_1}\exp\Big\{-2R_{3,j}\big[\frac{t}{\tau}\big]\Big\}|R_{3,j}|^2
\leq C\sum_{1\leq j\leq \frac{1}{4}\sqrt{\frac{1}{(2-\theta)\tau}}}(j^2t)^{-\alpha}j^4\tau^2
\leq C\tau^{\alpha-\frac{1}{2}}t^{-\alpha}.
\end{align*}
The term $\mathcal J_3=0$ in this setting, and the term $\mathcal J_4$ is estimated as \begin{align*}\mathcal J_4&=\sum_{j\in A_1}\int_{0}^{\infty}\exp\Big\{-2[\frac{t}{\tau}]\ln (1+R_{3,j})\Big\}\Big|1-\exp\Big\{[\frac{t}{\tau}](-R_{3,j}+\ln (1+R_{3,j}))\Big\}\Big|^2dt\\
&\leq C\tau^{\alpha-\frac12}t^{-\alpha}.
\end{align*}    
The term $\mathcal J_5$ is estimated as $\mathcal J_5\leq \sum_{j\in A_2}j^{-2\alpha}t^{-\alpha}\leq \int_{\frac{1}{2\pi}\sqrt{\frac{1}{2-\theta}}}^{[\frac n2]\sqrt {\tau}}y^{-2\alpha}dy\tau^{\alpha-\frac12}t^{-\alpha}\leq C\tau^{\alpha-\frac12}\tau^{-\alpha}$ for $\alpha>\frac12.$ 
The remaining term $\mathcal J_6$ can be estimated as follows. 
For $\theta\in[0,\frac{1}{2}]$, by the  boundedness condition on   $n^2\tau$ in Assumption \ref{Assumption3}, we have that for $\alpha>0$, 
\begin{align*}
\sum_{j\in A_2}(R_{1,j}R_{2,j})^{2[\frac{t}{\tau}]}
\leq \frac{1}{\sqrt{\tau}}\int_{\frac{1}{2\pi}\sqrt{\frac{1}{2-\theta}}}^{[\frac{n}{2}]\sqrt{\tau}}(1-\epsilon)^{2[\frac{t}{\tau}]}\,dy\leq \frac{C}{\sqrt{\tau}}e^{-2[\frac{t}{\tau}]\ln(1-\epsilon)^{-1}}\leq \frac{C}{\sqrt{\tau}}\big[\frac{t}{\tau}\big]^{-\alpha}.  
\end{align*}
For $\theta=1,$ we obtain that for $\alpha>0,\;t\ge \tau$,
\begin{align*}
&\sum_{\frac{1}{2\pi}\sqrt{\frac{1}{\tau}}<j\leq [\frac{n}{2}]}(1-\tau\lambda^n_j)^{-2[\frac{t}{\tau}]}
\leq \frac{1}{\sqrt{\tau}}\int_{\frac{1}{2\pi}}^{\infty}(1+16y^2)^{-2[\frac{t}{\tau}]}\,dy\\
&\leq \frac{1}{\sqrt{\tau}}(1+\frac{4}{\pi^2})^{-[\frac{t}{\tau}]} \int_{\frac{1}{2\pi}}^{\infty}(1+16y^2)^{-[\frac{t}{\tau}]}\,dy
\leq C\frac{1}{\sqrt{\tau}}\big[\frac{t}{\tau}\big]^{-\alpha}.
\end{align*} 

\textit{Case $2:$  $\theta\in(\frac{1}{2},1).$}
Define $G^{(n)}(t,x,y)=1+2\sum_{j=1}^{[\frac n2]}e^{-4\pi^2j^2t}\cos(2\pi j(x-y)).$ Similar to the estimate of $I^c_1$, we derive  \begin{align*}
\Big|\int_0^1(G(t,x,y)-G^{(n)}(t,x,y))u_0(\kappa_n(y))dy\Big|^2\leq C
\sum_{j=[\frac n2]+1}^{\infty}e^{-8\pi^2j^2t}\leq Cn^{1-2\alpha}t^{-\alpha}.
\end{align*} 
Then we proceed to estimate error between $G^{(n)}$ and $G^{n,\tau}_1.$  
We have the following decomposition of the error  \begin{align*}
&\int_{0}^{1}\big(G^{(n)}(t,x,y)-G^{n,\tau}_1(t,x,y)\big)u_0(\kappa_n(y))\,dy\\
=&\int_0^1\big(G^{(n)}(t,x,y)-G^{(n)}(t,\kappa_n(x),\kappa_n(y))\big)u_0(y)\,dy\\
&+
\int_{0}^{1}\big(G^{(n)}(t,\kappa_n(x),\kappa_n(y))-G^{n,\tau}_1(t,x,y)\big)u_0(y)\,dy\\
&+\int_{0}^{1}\big(G^{(n)}(t,x,y)-G^{n,\tau}_1(t,x,y)\big)\big(u_0(\kappa_n(y))-u_0(y)\big)\,dy
=:Q_0+Q_1+Q_2.
\end{align*} 
For the term $Q_0,$ we further split as $Q_0=Q_{0,1}+Q_{0,2},$ where 
\begin{align*}
&Q_{0,1}:=\int_0^1\big(G^{(n)}(t,x,y)-G^{(n)}(t,\kappa_n(x),y)\big)u_0(y)\,dy,\\
&Q_{0,2}:=\int_0^1\big(G^{(n)}(t,\kappa_n(x),y)-G^{(n)}(t,\kappa_n(x),\kappa_n(y))\big)u_0(y)\,dy. 
\end{align*}
Similar to the estimate of term $I^c_2+I^c_3,$ by the mean value theorem,  the term $Q_{0,1}$ can be estimated as $|Q_{0,1}|^2\leq C\sum_{i=1}^{[\frac n2]}\frac{j^2}{n^2}e^{-8\pi^2j^2t}\leq Cn^{1-2\alpha}t^{-\alpha}.$  Using  the mean value theorem again, the term $Q_{0,2}$ can be estimated as 
\begin{align*}
|Q_{0,2}|&\leq 
\sum_{j=1}^{[\frac n2]}e^{-4\pi^2j^2t}\int_{0}^{1}|\bar e_{j}(y)-\bar e_j(\kappa_n(y))||u_0(y)|dy\leq C\sum_{j=1}^{[\frac n2]}\frac jn e^{-4\pi^2j^2t}\\
&\leq C\sum_{j=1}^{[\frac n2]}\frac jn(j^2t)^{-\alpha}\leq Ct^{-\alpha}n^{-1}\int_0^{n}x^{1-2\alpha}dx\leq Cn^{1-2\alpha}t^{-\alpha},\;\alpha<1.
\end{align*} 
For the term $Q_1,$ noting that 
$e_{[\frac n2]+j}(\kappa_n(x))=\bar e_{[\frac n2]-j+1}(\kappa_n(x))$ and  $\lambda^n_{-j}=\lambda^n_j$ for $j=1,\ldots,[\frac  n2],$ we have 
\begin{align*}
&\quad G^{(n)}(t,\kappa_n(x),\kappa_n(y))-G^{n,\tau}_1(t,x,y)\\
&=\sum_{j=-[\frac n2]}^{[\frac n2]}(e^{-4\pi^2 j^2t}-(R_{1,j}R_{2,j})^{[\frac{t}{\tau}]})e_j(\kappa_n(x))\bar e_j(\kappa_n(y))\\
&=\sum_{j=0}^{[\frac{n}{2}]}(e^{-4\pi^2 j^2t}-(R_{1,j}R_{2,j})^{[\frac{t}{\tau}]})e_j(\kappa_n(x))\bar e_{j}(\kappa_n(y))\\
&\quad +\sum_{j=[\frac n2]+1}^{n-1}(e^{-4\pi^2 (n-j)^2t}-(R_{1,n-j}R_{2,n-j})^{[\frac{t}{\tau}]})e_j(\kappa_n(x))\bar e_{j}(\kappa_n(y)).
\end{align*}
Hence, when $u_0\in H^1$, it follows from \eqref{sob_rea} and the fact $\lambda^n_{n-j}=\lambda^n_j$ for $j=1,\ldots,[\frac  n2]$ that  
\begin{align*}
&|Q_{1}|^2\leq \|G^{(n)}(t,\kappa_n(x),\kappa_n(\cdot))-G^{n,\tau}_1(t,x,\cdot )\|^2_{H^{-1}}\|u_0\|^2_{H^1}\\
&\leq C\sum_{j=1}^{[\frac n2]}(1-\lambda^n_j)^{-1}|e^{-4\pi^2j^2t}-(R_{1,j}R_{2,j})^{[\frac{t}{\tau}]}|^2\\
&\leq C\sum_{j\in  A_1}|e^{-4\pi^2j^2 t}-(R_{1,j}R_{2,j})^{[\frac{t}{\tau}]}|^2+C\sum_{j\in  A_2}|e^{-4\pi^2j^2 t}-(R_{1,j}R_{2,j})^{[\frac{t}{\tau}]}|^2\frac{1}{j^2}\\&
=:Q_{1,1}+Q_{1,2}.
\end{align*}
The estimate of the term $Q_{1,1}$ is similar to that  of $\mathcal J_0+\cdots+\mathcal J_4$ in \textit{Case 1} and  thus is omitted. For the term $Q_{1,2},$
\begin{align*}
Q_{1,2} 
&\leq\sum_{\frac{1}{2\pi}\sqrt{\frac{1}{(2-\theta)\tau}}<j\leq [\frac n2]}Ce^{-8\pi^2j^2t}\times \frac{1}{j^2}+\sum_{\frac{1}{2\pi}\sqrt{\frac{1}{(2-\theta)\tau}}<j\leq [\frac n2]}C(1-\epsilon)^{2[\frac{t}{\tau}]}\times \frac{1}{j^2}\\
&\leq C\big(e^{-\frac{4}{\pi^2(2-\theta)}\times [\frac{t}{\tau}]}+(1-\epsilon)^{[\frac{t}{\tau}]}\big)\int_{\frac{1}{2\pi}\sqrt{\frac{1}{(2-\theta)\tau}}}^{[\frac n2]}\frac{1}{x^2}\,dx
\leq C\big[\frac{t}{\tau}\big]^{-\alpha}\sqrt{\tau},\;\alpha>0.
\end{align*} 

For the term $Q_{2},$ we split it as 
\begin{align*}
Q_{2}&=\int_0^1(G^{(n)}(t,\kappa_n(x),\kappa_n(y))-G^{n,\tau}_1(t,x,y))(u_0(\kappa_n(y))-u_0(y))dy\\
&+\int_0^1(G^{(n)}(t,x,y)-G^{(n)}(t,\kappa_n(x),\kappa_n(y)))(u_0(\kappa_n(y))-u_0(y))dy=:Q_{2,1}+Q_{2,2}.
\end{align*}
Since $u_0\in H^1\hookrightarrow \mathcal C^{0,\frac12},$ we derive  
\begin{align*}
&|Q_{2,1}|^2
\leq \frac{C}{n}\int_0^1\big|G^{(n)}(t,\kappa_n(x),\kappa_n(y))-G^{n,\tau}_1(t,x,y)\big|^2\,dy\\
&\leq \sum_{j\in  A_1}\frac{C}{n}|e^{-4\pi^2j^2t}-(R_{1,j}R_{2,j})^{[\frac{t}{\tau}]}|^2
+\sum_{j\in  A_2}\frac{C}{n}|e^{-4\pi^2j^2 t}-(R_{1,j}R_{2,j})^{[\frac{t}{\tau}]}|^2\\
&=:Q^{(1)}_{2,1}+Q^{(2)}_{2,1}. 
\end{align*} 
The term $Q^{(1)}_{2,1}$ is estimated similarly as before and we obtain $Q^{(1)}_{2,1} \leq C\tau^{\alpha-\frac{1}{2}}([\frac{t}{\tau}]\tau)^{-\alpha}$ for $\frac{1}{2}<\alpha<2$. For $j\in  A_2,$ we have $j^2\tau\ge C$ and thus  
\begin{align*}
Q^{(2)}_{2,1}\leq \frac{C(n-1)}{n}\big(e^{-C[\frac{t}{\tau}]}+(1-\epsilon)^{2[\frac{t}{\tau}]}\big)
\leq C(\big[\frac{t}{\tau}\big])^{-\alpha} \text{ for } \alpha>0. 
\end{align*} The  term $Q_{2,2}$ is estimated similarly to term $Q_{0,2}$, and we obtain that 
$
Q_{2,2}
\leq C\sum_{j=1}^{[\frac n2]}e^{-4\pi^2j^2t}\frac{j}{n}
\leq Ct^{-\alpha}n^{1-2\alpha}$
for $\alpha<1.$

Combining \textit{Cases 1-2}, we finish the proof.  
\end{proof}

\bibliographystyle{plain}
\bibliography{references.bib}

 \end{document}